\documentclass[reqno]{amsart}

\allowdisplaybreaks
\usepackage{graphicx}
\usepackage{amsfonts}
\usepackage{mathrsfs}
\usepackage{bbold}
\usepackage{fancyhdr}
\usepackage{amsthm}
\usepackage{cases}
\usepackage{subcaption}
\usepackage{amsmath,amssymb,bm}
\usepackage{tikz}
\usepackage{tkz-euclide}
\usepackage{enumitem}
\usepackage{xcolor}
\usepackage{upgreek}
\usepackage{empheq}
\allowdisplaybreaks[3]
\usepackage[
bookmarks=true,  %generates bookmarks for all entries in the "Table of Contents"
bookmarksnumbered=true, %includes chapter-/header-/section-/subsection-/... -
colorlinks=true, pdfstartview=FitV, linkcolor=red, citecolor=blue,
urlcolor=blue]{hyperref}

\def\bb0{\mathbb{0}}

\def\cA{\mathcal{A}}

\def\EE{\mathbb{E}}
\def\PP{\mathbb{P}}

\def\cF{\mathcal{F}}

\def\fI{\mathfrak{I}}
\def\cJ{\mathcal{J}}

\def\bL{\mathbb{L}}
\def\fL{\mathfrak{L}}

\def\bN{\mathbb{N}}

\def\RR{\mathbb{R}}
\def\cS{\mathcal{S}}

\def\cS{\mathcal{S}}
\def\bT{\mathbb{T}}

\def\Om{\Omega}

\def\LL{[-L,L]}

\newcommand\bP{\mathbf{P}}
\newcommand\HH{\mathbb{H}}

\def\al{{\alpha}}
\def\ls{{\lesssim}}
\def\gs{{\gtrsim}}

\def\be{{\beta}}

\def\Om{{\Omega}}
\def\al{{\alpha}}

\def\be{{\beta}}

\def\ph{{\varphi}}

\def\vare{{\varepsilon}}

\def\bP{{\mathbf{P}}}
\def\diam { {\rm{diam\,}}}

\definecolor{ao(english)}{rgb}{0.0, 0.5, 0.0}

\newcommand{\1}{{\bf 1}}

\newcommand{\blc}{\big(}
\newcommand{\brc}{\big)}
\newcommand{\Blc}{\Big(}
\newcommand{\Brc}{\Big)}

\newcommand{\Blk}{\Big[}
\newcommand{\Brk}{\Big]}
\newcommand{\lc}{\left(}
\newcommand{\rc}{\right)}
\newcommand{\lk}{\left[}
\newcommand{\rk}{\right]}
\newcommand{\lt}{\left }
\newcommand{\rt}{\right}

\newtheorem{Thm}{Theorem}
\newtheorem{Lem}[Thm]{Lemma}
\newtheorem{Propos}[Thm]{Proposition}
\newtheorem{Rmk}[Thm]{Remark}

\newtheorem{Def}[Thm]{Definition}

\theoremstyle{remark}
\numberwithin{equation}{section}

\begin{document}

\title{Stochastic wave equation with additive fractional noise:  Solvability and global H\"older continuity}

\author{Shuhui Liu}
\address{Department of Applied Mathematics, The Hong Kong Polytechnic University, Hong Kong SAR, China}
\email{shuhui.liu@polyu.edu.hk}

\author{Yaozhong Hu}
\address{Department of Mathematical and Statistical Sciences, University of Alberta, Edmonton, AB T6G 2G1, Canada}
\email{yaozhong@ualberta.ca}
\thanks{YH was supported by the NSERC discovery fund and a startup fund of University of Alberta. }

\author{Xiong Wang}
\address{School of Mathematics, Sun Yat-sen University, Guangzhou, 510275, China}
\email{xiongwang@ualberta.ca (corresponding)}
\date{}


\maketitle

\begin{abstract}
We determine the range of Hurst parameters that provide the necessary and sufficient conditions for the solvability, in $L^2(\Omega)$, of the stochastic wave equation: $ \frac{\partial^2  }{\partial t^2}u(t,x) =\Delta u(t,x)+\dot{W}(t,x)$,  where $\{  W(t,x),\ t\ge 0, x\in \RR^d\} $ is a fractional Brownian field with temporal Hurst parameter  $H_0\in[\tfrac12,1]$ and spatial Hurst parameters  $H_i\in(0,1)$ for $i=1,\cdots,d$. {In particular, the solvability condition exhibits a phase transition at $H_0 = 1$.} We also obtain the sharp growth rate and the sharp H\"older continuity of the solution on the real line in the case $H_0=1/2$.
\end{abstract}
%\subjclass[2010]{Primary 60H15; secondary 60H05, 60G15, 60G22}
%\date{}

\section{Introduction}
In this paper, we consider the following stochastic   wave  equation (SWE) in any dimension, driven by  additive   Gaussian noise, which is both fractional in time and space with Hurst parameters $H_0\ge 1/2$ and $ H=( H_1,  \cdots, H_d)\in (0, 1)^d$ respectively:
\begin{equation}\label{eq.SWE}
\begin{cases}
  \frac{\partial^2 u(t,x)}{\partial t^2}=\Delta u(t,x)+\dot{W}(t,x),\quad   t\ge 0,\quad x\in\RR^d\,, \\
  u(0,x)=u_0(x)\,,\quad\frac{\partial}{\partial t}u(0,x)=v_0(x)\,, \quad  x\in\RR^d
\end{cases}
\end{equation}
where $\Delta =\sum_{i=1}^d \frac{\partial ^2}{\partial x_i^2}$ is the Laplacian. 
In the above equation, $W(t,x)$ is a centered Gaussian random field  with covariance given by
\begin{equation}\label{CovW}
  \EE[W(t,x)W(s,y)]=R_{H_0}(t, s)\prod_{i=1}^d   R_{H_i}(x_i, y_i)\,, 
\end{equation}
where the function $R_{a  }(\xi , \eta )$ is defined as 
\begin{equation}
R_{a  }(\xi , \eta )=\frac12 \left(|\xi|^{2a}+|\eta|^{2a} -|\xi-\eta|^{2a}\right) \,,\quad a\in (0, 1),\ \  \xi, \eta\in \RR \,.
\end{equation}
Formally we write $\dot W(t,x)=\frac{\partial ^{d+1}}{\partial t\partial x_1\cdots \partial x_d} W(t,x)$, then the covariance of the noise is
\begin{equation}\label{CovdotW}
  \EE[\dot  W(t,x)\dot W(s,y)]=\phi_{H_0}(t-s) \prod_{i=1}^d  \phi_{H_i}(x_i, y_i)\,,
\end{equation}
where
\begin{equation}
   \phi_{a}(x - y)= \mathfrak{c}_a  |x-y|^{2a-2}  \,,\quad\text{with }\mathfrak{c}_a:=a(2a-1) \,, a\in (0, 1),\ \  x, y\in \RR \,.
\end{equation}

We denote the Green's function associated with \eqref{eq.SWE}, i.e., the wave kernel by $G_t(x-y)$ and we also  use the following notation
    \begin{align}\label{def.I_0}
		I_0(t,x):= &G_t\ast v_0(x)
		+\frac{\partial}{\partial t} G_t\ast u_0(x) \,.
%		\nonumber \\
%		=&\frac 12\int_{x-t}^{x+t}v_0(y)dy+\frac{1}{2}[u_0(x+t)+u_0(x-t)]\,,
 \end{align}
Then the  solution    to
\eqref{eq.SWE} can be written explicitly as
\begin{align}
     u(t,x)
%     &=\frac{\partial}{\partial t} G_t\ast u_0(x)+G_t\ast v_0(x)+G_t\circledast %\sigma(\cdot,\cdot,u)(x) \nonumber\\
     &=I_0(t,x) +\int_{0}^{t}\int_{\RR}G_{t-s}(x-y)W(ds,dy)  \label{eq.MildSol}
    \end{align}
if the above stochastic integral exists.  
To focus on the stochastic part, we assume $u_0=0$ and $v_0=0$. Thus,  the  resulting   solution is written as
\begin{equation}\label{LinearMildSol}
	u(t,x)=  \int_{0}^{t}\int_{\RR} G_{t-s}(x-y) W(ds,dy)
\end{equation}
if the above stochastic integral (with deterministic integrand) exists.

Our first main result in  this paper is to identify the necessary and sufficient conditions on the Hurst parameters $H_0$ and $H=(H_1,  \cdots, H_d)$  so that  the above stochastic integral \eqref{LinearMildSol}  exists  (as a mean-zero Gaussian with finite variance), thereby 
 	completely characterizing the solvability of equation \eqref{eq.SWE}. 
 
More specifically, our first main result in this paper  is encapsulated in the following theorem:
\begin{Thm}\label{t.iff}   
Let $H_0\in [\tfrac{1}{2},1]$ and $H_i\in (0,1)$ for $i=1,\cdots,d$. Denote $|H|:=H_1+\cdots+H_d$.
	The necessary and sufficient conditions for the existence of  \eqref{LinearMildSol} 
	as a finite variance Gaussian variable (namely, the solvability of \eqref{eq.SWE}) are as follows:\begin{equation}
		\begin{cases}
			|H|>d-1 &\qquad \hbox{if $H_0=1/2$}\,;\\
			|H|>d-2&\qquad \hbox{if $H_0=1 $}\,;\\
			|H|+H_0>d-1/2&\qquad \hbox{if $1/2<H_0<1 $}\,.\\
		\end{cases}\label{e.1.iff}
	\end{equation}
\end{Thm}

\begin{figure}[htb]
\centering
\begin{tikzpicture}[scale=1.05,>=stealth]
  % Axes
  \draw[->] (0,0) -- (6.8,0) node[below right] {$H_0$};
  \draw[->] (0,0) -- (0,7.2) node[left] {$|H|$};

  % Tick labels for H0 = 1/2 and 1 (schematic positions)
  \draw (2,0) -- (2,-0.08) node[below] {$\tfrac12$};
  \draw (6,0) -- (6,-0.08) node[below] {$1$};

  % Horizontal reference levels (labels use d symbolically)
  % Place |H| = d-1 at y=3.2 and |H| = d-2 at y=2.0 (schematic heights)
  \draw[densely dashed,gray!70] (0,6) -- (2,6);
  \node[anchor=east,gray!70] at (-0.1,6) {$d-1$};

  \draw[densely dashed,gray!70] (0,1.5) -- (6,1.5);
  \node[anchor=east,gray!70] at (-0.1,1.5) {$d-2$};
  
  \draw[densely dashed,gray!70] (0,3) -- (6,3);
  \node[anchor=east,gray!70] at (-0.1,3) {$d-\tfrac32$};

  % Oblique boundary: |H| + H0 = d - 1/2, drawn schematically from H0=1/2 to 1
  % Choose points so that at H0=1/2 (x=2) the height meets "d-1" (y=3.2),
  % and as H0->1- (x->6-) it approaches "d-1/2-1" placed here at y=2.5.
  \coordinate (A) at (2,6);       % (H0=1/2, |H|=d-1)
  \coordinate (B) at (6,3);       % (H0=1-,  |H|=d-1/2-1) schematic
  \draw[very thick] (A) -- (B);

  % Endpoints and discontinuity markers
  \fill (A) circle (1.6pt);         % closed dot at H0=1/2
  \draw[thick] (B) circle (1.6pt);  % open dot at the oblique limit H0->1-

  % True boundary at H0 = 1 is |H| = d-2: filled dot
  \fill (6,1.5) circle (1.6pt);

  % Labels for the three cases
  \node[anchor=west] at (0.25,6.35) {$|H|>d-1$};
  \node[anchor=west] at (0.8,4) {$|H|+H_0>d-\tfrac12$};
  \node[anchor=west] at (6.1,1.5) {$|H|>d-2$};

  % Shaded solvable region: above oblique line on (1/2,1), and above y=d-2 at H0=1
  % (i) a polygon above the oblique line
  \path[fill=blue!30, opacity=0.3]
    (2,6) -- (6,3) -- (6,6.8) -- (2,6.8) -- cycle;

  % (ii) a thin vertical strip at H0=1 above y=d-2
  \path[fill=red!30, opacity=0.3]
    (5.95,1.5) -- (6.05,1.5) -- (6.05,6.8) -- (5.95,6.8) -- cycle;
    
  % (iii) a thin vertical strip at H0=1 above y=d-1
  \path[fill=red!30, opacity=0.3]
    (1.95,6) -- (2.05,6) -- (2.05,6.8) -- (1.95,6.8) -- cycle;

  % Axis braces for clarity (optional)
  \draw[decorate,decoration={brace,amplitude=4pt},gray!60]
    (1.95,0.2) -- (6.05,0.2) node[midway,yshift=10pt,gray!60]{\(\tfrac12<H_0<1\)};
\end{tikzpicture}
\caption{The solvability region of Thorem \ref{t.iff}}
\end{figure}

{%\red 
The model \eqref{eq.SWE} has been extensively studied in the literature. For the SWE driven by (colored) additive noise, \cite{BalanSPA10} extended Dalang's random field framework \cite{DalangEJP99} from white-in-time noise to temporally fractional noise with $H_{0}\in(1/2,1)$ and spatially homogeneous fractional covariance with $H_{i}>1/2, i=1,\dots,d$. In particular, they established that the condition $|H| + H_{0} > d - \tfrac12$ in \eqref{e.1.iff} is necessary and sufficient for solvability. For the one-dimensional SWE driven by multiplicative noise, the results in \cite{BJQ2015,BJQSPL16,LuisSPA20,LHW2022} showed that $H > 1/4$ is the exact threshold for well-posedness. For nonlinear SWEs with additive noise, the works \cite{Deya19,DeyaSPA22} likewise demonstrated that the condition $|H| + H_{0} > d - \tfrac12$ in \eqref{e.1.iff} characterizes the existence of function-valued solutions. We also remark that, to the best of our knowledge, the discontinuity phenomenon at $H_{0}=1$ observed in this work has not been reported.}

The detailed proof of Theorem \ref{t.iff} is given in Section \ref{s.2}. The argument essentially reduces to determining whether a certain elementary multiple integral, such as \eqref{est1} in the next section, is convergent or not.  
    The cases $H_0=1/2$ (time white) or $H_0=1$ (time independent) are relatively straightforward. However, evaluating this multiple integral in the case $1/2<H_0<1$ is much more sophisticated since we aim to derive the necessary and sufficient condition. By complex computations, we reduce the analysis of this multiple integral to analyzing the behavior of another integral $g_1(\rho)$ (see equation \eqref{g_1} below) as $\rho$ tends to infinity.  By utilizing the asymptotic of the generalized hypergeometric function ${}_1F_2$, we are ultimately able to determine the exact range of the Hurst parameter for which the concerned integral converges. 
	 It is worth noting that,  if we formally let $H_0=1$ in the third condition of  \eqref{e.1.iff}, we get $|H|>d-3/2$, which differs from the second condition $|H|>d-2$ given there.  This reveals an interesting discontinuity of the solvability condition at $H_0 = 1$ in \eqref{e.1.iff}.  The reason for this discontinuity is that, as $\rho\to\infty$, $g_1(\rho)\asymp \rho$ when $1/2<H_0<1$, but $g_1(\rho)$ is no longer of the order $\rho$ when $H_0=1$. 
	 This gives an explanation of the discontinuity of the solvability condition at \eqref{e.1.iff} when $H_0 = 1$.

Let us mention that for the stochastic heat  equation (SHE)
with additive noise (when $\frac{\partial ^2}{\partial t^2}$  in \eqref{eq.SWE}  is replaced by $\frac{\partial  }{\partial t}$ )   a necessary and sufficient 
condition has been found in \cite{hlt2019}  
for quite a general class of Gaussian noises.
  Extending this result to   SWE presents a  challenge, primarily due to the oscillatory nature of the Fourier transform   $\hat G_t(\xi)=\frac{\sin (t|\xi|)}{|\xi|}$  of the wave kernel compared to     the Fourier transform for  the heat kernel, which is always positive. This oscillatory nature requires much more delicate analysis about the convergence and divergence of the concerned oscillatory integrals.   
  
For this reason, we assume in this paper that the noise is a fractional one, and the temporal Hurst parameter is assumed to be greater than or equal to   $1/2$, and the spatial Hurst parameters   $(H_1,\cdots,H_d)\in(0,1)^d$ can be arbitrary. 
 For this range of Hurst parameters, we  refer to two recent papers \cite{CH2022,HLW2022} that provide a necessary and sufficient condition for the parabolic Anderson model (SHE with multiplicative noise) to be solvable.

Upon establishing the well-posedness for equation \eqref{eq.SWE}, our next objective is to derive some precise properties for the solution 
$u(t,x)$.  Inspired by the results of \cite{HW2021}, we aim to ascertain the growth property and the temporal and spatial H\"older continuities of the solution on the entire $\RR^d$. More specifically, we want to know the sharp growth rate of $\sup_{0\le t\le T, |x|\le L} |u(t,x)|$ in terms of $T$ and $L$ as $T, L\to \infty$.  It is known that $u(t,x)$ is  H\"older continuous in $t$ and $x$.  Namely, {
%\blue 
there are $\al$ and $\beta$ such that
in any bounded domain $D\subseteq \RR_+\times \RR^d$}  \begin{equation*}
	|u(t,x)-u(s,y)|\le   C_D \left[|t-s|^\beta +|x-y|^\al\right]   \,,\quad \forall \ (t,x)\in D 
\end{equation*} 
for some finite positive (random) constant $C_D$.  We would like to determine the optimal exponents $\al$ and $\beta$.  Specifically, we seek
$\al$ and $\beta$ such that
\begin{equation}
	|u(t,x)-u(s,y)|\asymp   C_D \left[|t-s|^\beta +|x-y|^\al\right]  \label{e.1.12}
\end{equation} 
for some finite positive (random) constant $C_D$.
Obviously, the   constant $C_D$ should depend  on the domain
$D$. %\tcb
{We are particularly   interested in understanding dependence of  $C_D$ on the diameters  of $D$ as  $D$ approaches $\RR^d$, namely, the global H\"older continuity.}  More precisely, we only need to consider  the  domain of the form $D=[0, T]\times \{x\in  \RR^d\,; |x|\le L\}$ and we want to know how the constant $C_D$ 
grows as $T$ and $L$ go to infinity.    Since we are concerned with  the equivalence such as  \eqref{e.1.12} instead of    only the upper bound, which is much harder,
% than obtaining the usual H\"older continuity, 
we have   succeeded only in the case $d=1$ and $H_0=1/2$ 
(one dimensional and time white case) thus far. While the method is expected to apply to higher-dimensional settings and to a broader class of Gaussian noises, a detailed treatment of these cases is deferred to future work.

It is also interesting to take the expectation in \eqref{e.1.12}.  Thus, we have three results in this paper. The first result pertains to the growth rate, presented in both mean and almost surely forms, and is described in the following theorem.   

\begin{Thm}\label{LUEESup}
 Assume $d=1$ and $H_0=1/2$.
Let the Gaussian field $u(t,x)$    be the solution to \eqref{eq.SWE} with  $u_0(x)=0$ and $v_0(x)=0$.    Then, the following conclusions hold.
    \begin{enumerate}
    \item[(1)] There exist two (strictly) positive
    constants $c_H$ and $ C_H$, independent of $T$ and $L$,   such that
\begin{align}
c_H \, \Phi(T, L)&\le
\EE \lk\sup_{{0\le t\le T\atop  -L\le x\le L}}  u (t,x) \rk\nonumber\\
&\le \EE \lk\sup_{{0\le t\le T\atop  -L\le x\le L}} |u (t,x)|\rk\le C_H  \, \Phi(T, L)\,,
 \label{SupEEasmp}
\end{align}
where
\begin{align}\label{def_phi0}
\Phi_0(T, L):=
\begin{cases}
1+\sqrt{\log_2\lc L/T \rc}\,, &L> T \\
1, &L\leq T
\end{cases}
\end{align}
and
\begin{align}
\  \Phi(T, L) =
% \begin{cases}
T^{\frac{1}{2}+H}\Phi_0(T, L).
 %&\text{if $L> T$},  \\
%      T^{\frac{1}{2}+H} & \text{if $L < T$}\,.
%    \end{cases}
\label{e.def_rho}
\end{align}
 \item[(2)]
     There exist two (strictly) positive random constants $c_H$ and $C_H$, independent of $T$ and $L$,  such that almost surely
    \begin{align}\label{Supasmp}
    c_H\,\Phi(T, L)  &\le   \sup_{(t,x)\in\Upupsilon(T,L)}  u (t,x)   \\
    &\leq  \sup_{(t,x)\in\Upupsilon(T,L)} |u (t,x)|\le C_H\,
        \Phi(T, L)   \,, \nonumber
    \end{align}
  where $\Upupsilon(T,L)
  =\{(t,x)\in[0,T]\times[-L,L]~:~L> T\}$\,. 
    \end{enumerate}
\end{Thm}

{%\red  
To compare the above result concerning the corresponding result  for  SHE (i.e., $\partial_{tt}$ replaced by $\partial_t$ in model \eqref{eq.SWE}), it is worth noting that in \cite[Theorem 1.1]{HW2021}, $\Phi_0(T, L)=1+\sqrt{\lt(\log_2\lc L/\sqrt{T} \rc\rt)^+}$ and $\Phi(T, L)= T^{\frac{H}{2}}\Phi_0(T, L) $, which differentiate   from the corresponding quantities in current  Theorem \ref{LUEESup}. In addition, \cite{KKXAOP17} shows that for any $t>0$, $\limsup\limits_{|x|\to \infty}\frac{u(t,x)}{\|u(t,x)\|_{L^2(\Omega)}\sqrt{\log_2|x|}}=\sqrt{2}$ almost surely where $u$ is the solution to SHE. We also remark that the global spatial behavior of the solution of SHE/SWE driven by additive noise is closely related to the intermittency properties of parabolic/hyperbolic Anderson models studied in \cite{balan2022,BalanConusAOP16,BJQ2017,ConusAOP13,HW2022,KKXAOP17}.
}

% \end{Rmk}

{%\blue 
Next, we aim to prove the global H\"older continuity with exponent $H-\epsilon$ for any $\epsilon>0$ of the solution in the spatial variable for all $t> 0$.} 
\begin{Thm}\label{LUHolderEESup}  Assume $d=1$ and $H_0=1/2$.
	Let $u (t,x)$  be the solution to \eqref{eq.SWE} with $u_0(x)=0$ and $v_0(x)=0$. Denote
	\[
	\Delta _h u (t,x):=u (t,x+h)-u (t,x)\,.
	\]
 Then for any  given $0<\theta<H$, there are (strictly) positive constants $c$, $c_H$
	and $C_{H, \theta}$ such that    the following inequalities hold  true for all
	 $L> T>0$ and    $0<|h|\leq c(t\wedge1)$:
\begin{align}
  c_{H}\,t^{\frac12}|h|^{H} \Phi_0(t, L)&\leq \EE \lk\sup_{-L\le x\le L} \Delta _h u (t,x) \rk \label{SupHolderEEasmp}\\
       & \leq \EE \lk\sup_{-L\le x\le L} |\Delta _h u (t,x)|\rk\leq C_{H,\theta}\,  t^{H-\theta+\frac12} |h|^{\theta} \Phi_0(t, L).  \nonumber
	\end{align}
 Moreover,  there are two (strictly) positive  random constants $c_{H}$ and
    $C_{H,\theta} $ such that it holds almost surely
\begin{align}
	c_{H}\,  t^{\frac12}|h|^{H} \Phi_0(t, L)
	 &\leq \sup_{-L\le x\le L}   \Delta _h u (t,x)
\label{SupasmpHolder}  \\
	 &\leq \sup_{-L\le x\le L} \left| \Delta _h u (t,x) \right| \leq
    C_{H,\theta}\,  t^{H-\theta+\frac12} |h|^{\theta} \Phi_0(t, L)
\nonumber
	\end{align}
	for all
	$L> T>0$ and    $0<|h|\leq c(t\wedge1)$.
\end{Thm}

{%\blue 
We now present the final main result of our work, which concerns the global H\"older continuity with exponent $H-\epsilon$ for any $\epsilon>0$ of the solution in the time variable over the entire space $x\in\RR$ for the solution.} 
\begin{Thm}\label{LUtHolderEESup}   Suppose $d=1$ and $H_0=1/2$.
Let $u (t,x)$    be the solution to \eqref{eq.SWE} with  $u_0(x)=0$ and $v_0(x)=0$  and denote
	\begin{equation*}
	\Delta_\tau u(t,x):=u (t+\tau,x)-u(t,x) \,.
	\end{equation*}
Then for any given $0<\theta<H$, there exist (strictly) positive constants  $c$, $c_H$
	and $C_{H, \theta}$ such that
\begin{align}	
 c_{H}\, t^{1/2}\tau^{H}\Phi_0(\tau,L)&\leq \EE \lk\sup_{-L\le x\le L}  \Delta_\tau u(t,x) \rk\label{SuptHolderEEasmp}\\
    &\leq \EE \lk\sup_{-L\le x\le L} |\Delta_\tau u(t,x)|\rk\leq
     C_{H,\theta}t^{1/2}\tau^{\theta} \Phi_0(\tau,L) \nonumber
\end{align}
for $L\geq\tau>0$  and    $0<\tau\leq c( {t}\wedge1)$.
Furthermore, we have the almost sure version of the above result. This is,
\begin{align}		
c_{H}\, t^{1/2}\tau^{H}\Phi_0(\tau,L) &\leq \sup_{-L\le x\le L} \Delta_\tau u(t,x)
\label{SuptasmpHolder}\\
&\leq \sup_{-L\le x\le L} |\Delta_\tau u(t,x) |\leq
    C_{H,\theta}t^{1/2}\tau^{\theta} \Phi_0(\tau,L) \nonumber
\end{align}
holds almost surely for all $L\geq \tau>0$ and    $0<\tau\leq c( {t}\wedge1)$,  where  $c$ is a positive constant, $c_H$
	and $C_{H, \theta}$ are  two random positive constants.
\end{Thm}

The above  four inequalities 
\eqref{SupHolderEEasmp}-\eqref{SuptasmpHolder}
are sharp since compared to the Brownian motion case, we believe that we can only allow   $\theta <H$ to 
be arbitrarily close to $H$  
on the right hand sides but usually it is impossible 
to allow $\theta=H$. This fact imposes that on the left
hand side we must take $\theta=H$. 

To prove the above results (Theorems \ref{LUEESup}-\ref{LUtHolderEESup}), we shall apply Talagrand's majorizing measure theorem and Sudakov's minoration theorem. This requires us to get the precise (matching) upper and lower bounds of the corresponding canonical metric, denoted as $d_1((t,x),(s,y)) = \sqrt{\mathbb{E}|u(t,x) - u(s,y)|^2}$,  associated with the solution $u(t,x)$. The analysis of these bounds differs from that of the SHE (\cite{HW2021}) in many aspects. A notable difficulty arises from the lack of monotonicity in the Fourier transform of the wave kernel. This obstacle is effectively solved by delving into integrals across infinitely varied intervals. The efficacy of our approach lies in carefully considering the integral over diverse intervals, circumventing the non-monotonicity issue inherent in the Fourier transform of the wave kernel. The detailed analysis is presented in Section 3 below.  

{%\red 
For nonlinear SWE with rough noise, it has been shown in \cite{BJQSPL16, HuAOP2017} that the solution admits a modification which is H\"older continuous of order $(H-\epsilon)$ in both time and space, for any $\epsilon > 0$. In particular, they proved that the $p$-moments of $\Delta _\tau u (t,x)$ and $\Delta _h u (t,x)$ can be bounded above by constants of order $|\tau|^H$ and $|h|^H$, respectively. Under   Dalang's condition for fractional Brownian field ($H_0, H_j > 1/2$, $j=1,\cdots,d$), similar H\"older regularity results for both nonlinear SHE and SWE have been well established; see, for example, \cite{dalang2009, DalangBEJ10, delgado2020, hu2014, QT2007}.

Furthermore, the papers \cite{LX2019, LXSPA23} investigate the exact modulus of continuity for the SWE under Dalang's condition. They establish that there exists a finite positive constant $K$ such that:
	\[
	\lim_{\epsilon \to 0+} \sup_{\substack{(t,x),(t',x') \in [a,a']\times [-b,b]^{d} \\ \sigma[(t,x),(t',x')] \le \epsilon}} \frac{|u(t,x) - u(t',x')|}{\gamma[(t,x),(t',x')]} = K \quad \text{a.s.}
	\]
	where the modulus function $\gamma$ is defined by the canonical metric $\sigma[(t,x),(s,y)]^2\\ =\EE[|u (t,x)-u(s,y)|^2]$ (which is $d_1$ in \eqref{NaturalMetric} below in one-dimensional setting) and a logarithmic correction:
	\[\gamma(\cdot) = \sigma(\cdot) \sqrt{\log(1 + \sigma(\cdot)^{-1})}\,.\]
Similar results for the SHE have been obtained in \cite{HSWX2020}.
}

To compare the results in the above-mentioned references with those obtained in this work, we stress the following three points. 
First, our analysis accommodates spatial Hurst parameters {%\blue 
$H_1$ that may be smaller than $1/2$}. 
Second, we establish matching lower bounds for the H\"older continuity exponents in \eqref{SupHolderEEasmp}-\eqref{SuptasmpHolder}, showing that the temporal and spatial H\"older exponents obtained in this work are indeed sharp.  
The third one is that we find the explicit global dependence of the H\"older constants on the 
diameters of domain.   
To the best of our knowledge,  when the spatial parameters are rough,  there is no necessary and sufficient condition on the solvability of the SWE and there is no result on the exact H\"older continuity of the solution. In particular, there has been no result on the global  H\"older continuity of the solution when the considered domain grows to infinity.  As in \cite{HW2021}, these results were critical to determining the solution space for the general nonlinear SHE, and we expect these results to also be needed to study the general nonlinear SWE.

The paper is organized as follows. Section 2
gives the proof of Theorem \ref{t.iff}.  
Section 3 gives  proofs of  Theorems  \ref{LUEESup}-\ref{LUtHolderEESup}  after 
obtaining the precise bound for the %\tcb
{canonical metric} 
$d_1((t,x),(s,y)) = \sqrt{\mathbb{E}|u(t,x) - u(s,y)|^2}$,  associated with the solution $u(t,x)$.   
In \ref{appenA}, we summarize the main theorems employed in this paper, and in \ref{appenB}, we present several technical proofs that were omitted from Section~\ref{s.2}.

%\tcb
{Throughout this paper, we write $A\approx B$ to indicate that there exists a nonzero  constant $C_1$ such that $|A-C_1B|=o(B)$.  We use   $A\ls B$ (or $A\gs B$)  to mean that there exists universal constants $C_1,\ C_2\in (0,\infty)$ such that $A\leq C_1B$ (or $A\geq C_2 B$). The notation $a\wedge b$ means the minimum of $a$ and $b$. Similarly, the notation $a\vee b$ means the maximum of $a$ and $b$.}

\section{Sufficient and necessary conditions
}\label{s.2}

In this section, we begin with an overview of some preliminary concepts. Subsequently, we provide the proof for Theorem \ref{t.iff}. Additionally, we elucidate the discontinuity observed in conditions \eqref{e.1.iff} when $H_0=1$.

 For any $\ph\in L^1(\RR^d)$, let $\cF\ph$ denote the Fourier transform of $\ph$ given by:
\[\cF\ph(\xi)=\int_{\RR^d}e^{-i\xi\cdot x}\ph(x)dx.\]
The Hilbert space $\HH$ is defined as the completion of the Schwartz space $\cS(\RR_+,\RR)$ concerning the inner product in spatial Fourier mode, as expressed below:
\begin{equation}\label{four_mode}
\langle f, g\rangle_\HH=C_{H}\int_{\RR_{+}^2\times \RR^{ d}} |r-s|^{2H_0-2}\cF f(r,\xi) {%\blue 
\overline{\cF g(s,\xi)}}\prod_{k=1}^d|\xi_k|^{1-2H_k}drdsd\xi,
\end{equation}
where $C_H=\frac{\Gamma(2H+1)\sin(\pi H)}{2\pi}$. Proceeding, we give the stochastic integration with respect to $W$, commencing with the integration of elementary processes.
\begin{Def}\label{Ele_def}
For $t\geq 0$, an elementary process $g$ is $\cF_t$-adapted random process given by  the following form:
\[
g(t,x)=\sum_{i=1}^n\sum_{j=1}^mX_{i,j}\1_{(a_i,b_i]}(t)\1_{(c_j,d_j]}(x)\,,
\]
where $n$ and $m$ are  positive and finite integers, $0\le a_1< b_1 < \cdots<a_n< b_n<+\infty$, $c_j<d_j$ and $X_{i,j}$ are $\mathcal{F}_{a_i}$-measurable random variables for $i=1,\cdots,n, j=1, \cdots, m$. The  stochastic integral of such a process $g$ with respect to $W$ is defined as
\begin{align}\label{elem_def}
\int_{\mathbb{R}_+\times \mathbb{R}^d}&g(t,x)W(dt,dx)=\sum_{i=1}^n\sum_{j=1}^mX_{i,j}W\lt(\1_{(a_i,b_i]}\otimes\1_{(c_j,d_j]}\rt)\nonumber\\
&=\sum_{i=1}^n\sum_{j=1}^mX_{i,j}\lt[W(b_i,d_j)-W(a_i,d_j)-W(b_i,c_j)+W(a_i,c_j) \rt].
\end{align}
\end{Def}
The integration with respect to $W$ can be extended to a broader class of adapted processes (cf. \cite{BJQ2015,HuAOP2017}).
\begin{Propos}\label{Int_Gen}
Let $\Lambda_{H}$ be the space of adapted random {%\blue 
processes} defined on $\mathbb{R}_+\times\mathbb{R}$ such that $g\in\HH$ almost surely and $\mathbb{E}[\|g\|_{\HH}^2]<\infty$. Then, we have the following statements.
\begin{enumerate}
\item The space of elementary process defined in Definition \ref{Ele_def} is dense in $\Lambda_{H}$;
\item For $g\in\Lambda_{H}$, the stochastic integral $\int_{\mathbb{R}_+\times\mathbb{R}^d}g(t,x)W(dt,dx)$ is defined as the $L^2(\Omega)$-limit of stochastic integrals of elementary processes which approximates $g(t,x)$ in $\Lambda_{H}$. We have the following isometry equality  for this kind of stochastic integral
    $$\mathbb{E}\lt[\lt(\int_{\mathbb{R}_+\times\mathbb{R}^d}g(t,x)W(dt,dx)\rt)^2 \rt]=\mathbb{E}[\|g\|_{\HH}^2].$$
\end{enumerate}
\end{Propos}

Before giving the proof of Theorem \ref{t.iff}, we approximate the noise by a more regular one so that the corresponding equation does have a solution.
Let $p_\vare(x)=\frac{1}{(2\pi \vare)^{d/2}} e^{-\frac{|x|^2}{2\vare}}$
be the heat kernel and
consider
\begin{equation}
\dot W_\vare(t, x)= \int_{\RR^d}   p_\vare(x-y) \dot W( t ,  y) dy\,.
\end{equation}
The covariance  of $\dot W_\vare(t,x)$ is
then
\begin{equation}\label{CovdotW}
  \EE[\dot  W_\vare(t,x)\dot W_\vare (s,y)]=\phi_{H_0}(t-s)
  \int_{\RR^{2d}}\prod_{i=1}^d  \phi_{H_i}( z_i,  \zeta_i)p_\vare(x-z)p_\vare(y-\zeta)dzd\zeta\,.
\end{equation}
With $\dot W_\vare$ we approximate the equation \eqref{eq.SWE} by the following
stochastic wave equation
\begin{equation}\label{eq.SWE_approximate}
\begin{cases}
  \frac{\partial^2 u_\vare (t,x)}{\partial t^2}=\Delta u_\vare(t,x)+\dot{W}_\vare (t,x),\quad   t\ge 0,\quad
   x\in\RR^d\,, \\
  u_\vare(0,x)=0\,,\quad\frac{\partial}{\partial t}u_\vare(0,x)=0\,.
\end{cases}
\end{equation}
We denote Green's function for the wave operator associated with \eqref{eq.SWE} 
(or \eqref{eq.SWE_approximate})   by
$G_t(x-y)$,  which  has the following well-known form when $d=1, 2, 3$:
\begin{equation}
	G_t(x)=\begin{cases}
		\frac{1}{2} \1_{\{|x|< t\}}&\qquad \hbox{when $d=1$}\,,\\
		\frac{1}{2\pi}\frac{1}{\sqrt{t^2-|x|^2}}{
		%%\blue 
		\1_{B(0,t)}(x)} &\qquad \hbox{when $d=2$}\,,\\
		\frac{1}{4\pi t}\sigma_t(dx)&\qquad \hbox{when $d=3$}\,,
	\end{cases}
\end{equation}
{%\blue 
where $\sigma_t(dx)$ denotes  the uniform measure on sphere $\partial B(0,t)$ centered at $0$ with radius $t$.} 
Green's function in the higher dimensional case is more complicated; however, its Fourier transform has the following consistent form for all dimensions:
\begin{equation}
	\hat G_t(\xi)=\frac{\sin (t|\xi|)}{|\xi|}\,,
	\quad t\ge 0\,, \ \xi \in \RR^d\,.
\end{equation}

It is easy to see that for any  $H_0\ge 1/2$ and for any $H\in (0, 1)^d$,
\[
u_\vare(t,x)=\int_0^t \int_{\RR^d} G_{t-s}(x-y) W_\vare(ds, dy)
\]
exists in the sense of Proposition \ref{Int_Gen}  as a finite variance Gaussian random field.
\begin{Def} If $\{  u_\vare(t,x)\,, \vare\to 0\}$ is a Cauchy sequence in $L^2(\Om, \cF, \PP)$ for any
$(t,x)\in \RR_+\times \RR^d$,  then we say \eqref{eq.SWE} is solvable, and the limit is called its solution.
\end{Def}

\begin{proof}[Proof of Theorem \ref{t.iff}]\
By Proposition \ref{Int_Gen} and Plancherel's identity, we have {%\blue 
for $H_0\in [1/2,1]$ and $H=(H_1,\cdots,H_d)\in(0,1)^d$},
\begin{align}\label{est1}
\EE[ u_\vare (t,x)
u_{\vare'}  (t,x) ]&=\EE\lt[\lt|\int_{0}^{t}\int_{\RR^d} G_{t-s}(x-y)
 W_\vare(ds,dy)\rt|^2\rt]\nonumber\\
&=\int_{[0,t]^2}\int_{\RR^d}\cF[G_{t-s}(x-\cdot)](\xi)\overline{\cF[G_{t-r}(x-\cdot)](\xi)}
\nonumber\\
&\qquad \cdot\Lambda_{H_0}(r-s)\cdot\prod_{i=1}^d|\xi_i|^{1-2H_i}e^{-\frac{(\vare +\vare') |\xi|^2}
{2}}d\xi dsdr\nonumber\\
&=\int_{[0,t]^2}\int_{\RR^d}\frac{\sin(s|\xi|)\cdot\sin(r|\xi|) }{|\xi|^2}
\nonumber\\
&\qquad \cdot\Lambda_{H_0}(r-s)\cdot\prod_{i=1}^d|\xi_i|^{1-2H_i}e^{-\frac{(\vare +\vare') |\xi|^2}
{2}}d\xi dsdr\,.
\end{align}
where {%\blue 
$\Lambda_{H_0}(r-s)=\mathfrak{c}_{H_0}|r-s|^{2H_0-2}$ if $H_0\in(1/2,1)$, $\Lambda_{1/2}(r-s)=\delta(r-s)$ and $\Lambda_{1}(r-s)=1$.}
Denote
\[
I_{\vare, \vare'}:=\EE[ u_\vare (t,x)
u_{\vare'}  (t,x) ].
\]
Then we know that $\{u_\vare (t,x)\}_{\vare\geq 0} $ is Cauchy in $L^2(\Om, \cF, \PP)$ if $I_{\vare, \vare'}$ is convergent.

In the following, we shall divide the discussion of the  convergence of $I_{\vare, \vare'}$ into three cases according to the values of the temporal Hurst parameter: $H_0=1/2$, $H_0=1$, $0<H_0<1/2$.\\

\noindent \textbf{Step 1: the case $H_0=1/2$.} In this case, equation \eqref{est1} takes the form
\begin{align*}
I_{\vare, \vare'} 
&=\int_0^t \int_{\RR^d}\frac{\sin^2(s|\xi|) }{|\xi|^2}\cdot\prod_{i=1}^d|\xi_i|^{1-2H_i} e^{-\frac{(\vare +\vare') |\xi|^2}{2} }  d\xi ds\\
&=\int_{\RR^d}\frac{1 }{|\xi|^2}\cdot\lt[\frac t2-\frac{\sin(2t|\xi|)}{4|\xi|}\rt]\cdot\prod_{i=1}^d|\xi_i|^{1-2H_i}  e^{-\frac{(\vare +\vare') |\xi|^2}{2} }  d\xi \\
&=:\int_{\RR^d}f_{(\vare,\vare')}(t,\xi,H)d\xi .
\end{align*}
It is clear that  as $\vare, \vare'\to 0$
\[
f_{(\vare,\vare')}(t,\xi,H)\rightarrow\frac{1 }{|\xi|^2}\cdot\lt[\frac t2-\frac{\sin(2t|\xi|)}{4|\xi|}\rt]\cdot\prod_{i=1}^d|\xi_i|^{1-2H_i} 
\]  
and $f_{(\vare,\vare')}(t,\xi,H)$ is also dominated by the above limiting quantity for any $\vare,\vare'\geq 0$. Thus, by the Lebesgues  dominated convergence theorem, if we can show
\begin{equation}
I:=\int_{\RR^d}\frac{1 }{|\xi|^2}\cdot\lt[\frac t2-\frac{\sin(2t|\xi|)}{4|\xi|}\rt]\cdot\prod_{i=1}^d|\xi_i|^{1-2H_i}     d\xi <\infty, \label{e.I.half}
\end{equation}
then, it holds that
\[
\lim_{\vare, \vare'\to 0}I_{\vare, \vare'} =I\,.
\]
This is to say, $\{u_\vare (t,x)\}_{\vare\geq 0}$ is Cauchy in $L^2(\Om, \cF, \PP)$.

{ 
We shall use the spherical coordinates to estimate the integral in \eqref{e.I.half} (e.g., see also \cite[(7.5)]{HW2022}):
\[
\left\{\begin{split}
\xi_1 &= \rho  \cos(\varphi_1) \\
\xi_2 &= \rho \sin(\varphi_1) \cos(\varphi_2) \\
\xi_3 &= \rho \sin(\varphi_1) \sin(\varphi_2) \cos(\varphi_3) \\
    &\,\,\,\vdots\\
\xi_{d-1} &= \rho \sin(\varphi_1) \cdots \sin(\varphi_{d-2}) \cos(\varphi_{d-1}) \\
\xi_d     &= \rho \sin(\varphi_1) \cdots \sin(\varphi_{d-2}) \sin(\varphi_{d-1}) \,,
\end{split} \right.
\]
where $0\le \rho <\infty\,, 0\le \varphi_1, \cdots, \varphi_{d-2}\le \pi  \,,
0\le \varphi_{d-1}\le 2\pi$, and whose Jacobian is
\begin{align*}
|J_d| &= \rho^{d-1}\sin^{d-2}(\varphi_1)\sin^{d-3}(\varphi_2)\cdots \sin(\varphi_{d-2})\,.
\end{align*}
For $t\geq 0$ and $|\xi|\geq 1$, rewriting the integral in spherical coordinates yields
\begin{align*}
	\int_{|\xi|\geq 1} &\frac{1}{|\xi|^2}\cdot\lt[\frac t2-\frac{\sin(2t|\xi|)}{4|\xi|}\rt]\cdot\prod_{i=1}^d|\xi_i|^{1-2H_i} d\xi \\
	&\lesssim \int_1^\infty \int_{[0,\pi]^{d-2}}\int_0^{2\pi} \rho^{-2} \rho^{\sum_{i=1}^d (1-2H_i)} \rho^{d-1} \prod_{i=1}^{d-1}d\varphi_i d\rho \\
	&\lesssim \int_1^\infty \rho^{\sum_{i=1}^d (1-2H_i)+d-3}d\rho
\end{align*}
which is finite if and only if
\begin{align*}
	\sum_{i=1}^d (1-2H_i)+d-3<-1 \ \Leftrightarrow \ |H|>d-1\,.
\end{align*}
Moreover, for $t\geq 0$ and $|\xi|< 1$, we have that
\begin{align*}
	\int_{|\xi|< 1} &\frac{1}{|\xi|^2}\cdot\lt[\frac t2-\frac{\sin(2t|\xi|)}{4|\xi|}\rt]\cdot\prod_{i=1}^d|\xi_i|^{1-2H_i} d\xi  
    \lesssim \int_0^1 \rho^{\sum_{i=1}^d (1-2H_i)+d-1}d\rho
\end{align*}
which is finite if and only if $|H|<d$. But this clearly holds since $H_i<1$ for any $i=1,\cdots,d$.
}

Thus, we can see from \eqref{e.I.half} that $I<\infty$  if and only if $|H|>d-1$, which proves the theorem when $H_0=1/2$.\\

\noindent \textbf{Step 2: the case $H_0=1$.} In this case, \eqref{est1} reduces to
\begin{align*}
I_{\vare, \vare'} &=\int_{[0,t]^2}\int_{\RR^d}\frac{\sin(s|\xi|)\cdot\sin(r|\xi|) }{|\xi|^2}\cdot\prod_{i=1}^d|\xi_i|^{1-2H_i} e^{-\frac{(\vare +\vare') |\xi|^2}{2} } d\xi dsdr\\
&=\int_{\RR^d}\frac{1}{|\xi|^4}\cdot[\cos(t|\xi|)-1]^2\cdot\prod_{i=1}^d|\xi_i|^{1-2H_i}e^{-\frac{(\vare +\vare') |\xi|^2}{2} }d\xi\,.
\end{align*}
As in the previous case, we can show that $\{u_\vare (t,x)\}_{\vare\geq 0} $ is Cauchy in $L^2(\Om, \cF, \PP)$ provided
\begin{align*} 
I=\int_{\RR^d}\frac{1}{|\xi|^4}\cdot[\cos(t|\xi|)-1]^2\cdot\prod_{i=1}^d|\xi_i|^{1-2H_i} d\xi<\infty \,.
\end{align*}
{ 
Similarly, one can utilize the spherical coordinates for $t\geq 0$ and $|\xi|\geq 1$:
\begin{align*}
	\int_{|\xi|\geq 1}\frac{1}{|\xi|^4}\cdot[\cos(t|\xi|)-1]^2\cdot\prod_{i=1}^d|\xi_i|^{1-2H_i} d\xi
	\lesssim \int_1^\infty \rho^{\sum_{i=1}^d (1-2H_i)+d-5} d\rho 
\end{align*}
which is finite if and only if
\begin{align*}
	\sum_{i=1}^d (1-2H_i)+d-5<-1\ \Leftrightarrow \ |H|>d-2\,.
\end{align*}
Furthermore, for $t\geq 0$ and $|\xi|< 1$ we have
\begin{align*}
	\int_{|\xi|< 1}\frac{1}{|\xi|^4}\cdot[\cos(t|\xi|)-1]^2\cdot\prod_{i=1}^d|\xi_i|^{1-2H_i} d\xi
	\lesssim \int_0^1 \rho^{\sum_{i=1}^d (1-2H_i)+d-1}d\rho
\end{align*}
which is finite if and only if $|H|<d$; this condition holds since $H_i<1$ for any $i=1,\cdots,d$.

Therefore, $I<\infty$ in this case if and only if $|H|>d-2$, which completes the proof for $H_0=1 $.\\
}

\noindent \textbf{Step 3: the case $H_0\in(\tfrac12,1)$.} In this case, by a change of  variables  $s|\xi|\to s$ and $ r|\xi|\to r$, equation \eqref{est1} can be rewritten as
\begin{align}
I_{\vare, \vare'} 
&= \int_{\RR^d} \int_0^{t|\xi|}\int_0^{t|\xi|} \frac{\sin( s)\cdot\sin( r) }{|\xi|^{2+2H_0}}\cdot| r- s|^{2H_0-2}\nonumber\\
&\qquad\qquad \cdot\prod_{i=1}^d|\xi_i|^{1-2H_i}  e^{-\frac{(\vare +\vare') |\xi|^2}{2} } dsd r d\xi.\label{u2}
\end{align}
Applying the spherical coordinates, we obtain that
\begin{align}\label{u1}
I_{\vare, \vare'}
&= C_{H_0, H}\int_0^\infty  \rho^{2d-2|H|-2H_0-3}
e^{-\frac{(\vare +\vare') \rho^2 }{2} } \nonumber\\
& \qquad\qquad \times\int_0^{t \rho }\int_0^{t\rho}  \sin( s)\cdot\sin( r)    \cdot| r- s|^{2H_0-2}  dsd r d\rho\nonumber\\
&= C_{H_0, H} t^{2+2|H|+2H_0-2d} \int_0^\infty  \rho^{2d-2|H|-2H_0-3}  e^{-\frac{(\vare +\vare') \rho^2 }{2} } \nonumber\\
& \qquad\qquad \times\int_0^{  \rho }\int_0^{ \rho}  \sin( s)\cdot\sin( r)    \cdot| r- s|^{2H_0-2}  dsd r d\rho\nonumber\\
&= C_{H_0, H} t^{2+2|H|+2H_0-2d} \int_0^\infty  \rho^{2d-2|H|-2H_0-3}
e^{-\frac{(\vare +\vare') \rho^2 }{2} }    g(\rho)  d\rho \,,
\end{align}
where $C_{H_0, H}$ is a finite positive constant depending only on $H_0$ and $H=(H_1,\cdots,H_d)$, and 
\[
g(\rho):=\int_{0< s<r<\rho}\sin( s)\cdot\sin( r) \cdot|r- s|^{2H_0-2}d sd r.
\]
%We shall prove that $|g(x)|\leq C x$ for some positive constant $C$ when $x\to +\infty$.
As in the previous two cases, we see that $\{u_\vare (t,x)\}_{\vare} $ is Cauchy in $L^2(\Om, \cF, \PP)$ if and only if 
\begin{align}\label{inte_g1}
\int_0^\infty  \rho^{2d-2|H|-2H_0-3}
e^{-\frac{(\vare +\vare') \rho^2 }{2} }    g(\rho)  d\rho 
&=\int_0^1 \rho^{2d-2|H|-2H_0-3}
e^{-\frac{(\vare +\vare') \rho^2 }{2} } g(\rho)  d\rho \nonumber\\
&+\int_1^\infty  \rho^{2d-2|H|-2H_0-3}
e^{-\frac{(\vare +\vare') \rho^2 }{2} }    g(\rho)  d\rho <\infty\,.
\end{align}
When $\rho\leq 1$, we have
\begin{align*}
	g(\rho) & \lesssim \int_{0< s<r<\rho}sr \cdot|r- s|^{2H_0-2}d sd r \approx \rho^{2H_0+2}\,.
\end{align*} 
Therefore
\[
\int_0^1 \rho^{2d-2|H|-2H_0-3}
e^{-\frac{(\vare +\vare') \rho^2 }{2} } g(\rho)  d\rho<+\infty
\]
if and only if $2d-2|H|-1>-1$, which holds obviously. Thus, from \eqref{inte_g1} we know
\[
\int_0^\infty  \rho^{2d-2|H|-2H_0-3}
e^{-\frac{(\vare +\vare') \rho^2 }{2} }    g(\rho)  d\rho<+\infty
\]  
is  equivalent to
\begin{equation}\label{I_finite}
\int_1^\infty  \rho^{2d-2|H|-2H_0-3}
     g(\rho)   d\rho<+\infty.
\end{equation}
Our goal in the following is to prove that \eqref{I_finite} holds if and only if
\[
|H|+H_0>d-1/2\,.
\]
To complete this task, we must find the exact asymptotics of $g(\rho)$ as $\rho\to \infty$. 
% {\red  A modification is added according to our discussion on Oct 18th. (Both reviewers have concerns about this argument, especially on page 15 of the original submission.)}

Let $\tilde s=r- s$ and $\tilde r= r+ s$. By elementary trigonometric identity
\[
{\sin(s)\sin(r) }=\frac{\cos(r-s)-\cos(s+r)}{2},
\]  
we obtain
\begin{align}
g(\rho)&=\frac 12\int_0^{\rho}\int_{\tilde s}^{2\rho-\tilde s}\frac{(\cos\tilde s)-(\cos\tilde r)}{2} \cdot|\tilde s|^{2H_0-2}d \tilde rd \tilde s\nonumber\\
&=\frac 12\int_0^{\rho}(\rho-  s)\cos  (s)\cdot  s ^{2H_0-2}d  s-\frac14\int_0^{\rho}[\sin(2\rho-  s)-\sin(  s)]\cdot  s ^{2H_0-2}d  s\nonumber\\
&=:\frac 12 g_1(\rho)-\frac14 g_2(\rho)\,.\label{compu1}
\end{align}
Let us first deal with the integral  $g_2(\rho)$. We write
\begin{align*}
g_2(\rho)&=  \int_0^1 [\sin(2\rho-  s)-\sin(  s)]\cdot  s ^{2H_0-2}d  s \nonumber\\
&+\int_1^{\rho} [\sin(2\rho-  s)-\sin(  s)]\cdot  s ^{2H_0-2}d  s \nonumber\\
&=:g_{21}(\rho)+g_{22}(\rho)\,.
\end{align*}
It is obvious that $g_{21}(\rho)$ is a bounded function of $\rho$ by recalling that $H_0\in (\tfrac{1}{2},1)$. 
For the term $g_{22}(\rho)$, integration by parts implies
\begin{align*}
g_{22}(\rho)&= \int_1^\rho     s ^{2H_0-2}
d [\cos(2\rho-  s)+\cos(  s) -2\cos(\rho)]\\
&=[\cos(2\rho-  s)+\cos(  s) -2\cos(\rho)]\cdot  s ^{2H_0-2}\big|_{s=1}^{s=\rho}\\
&\qquad\qquad -(2H_0-2) \int_1^{\rho}[\cos(2\rho-  s)+\cos(  s) -2\cos(\rho)]\cdot  s ^{2H_0-3}d  s
\end{align*}
which shows that $g_{22}(\rho)$ is bounded as well.
Thus $g_2(\rho)$ is bounded for $\rho\in (0, \infty)$.
% On the other hand, when $\rho\rightarrow 0+$, $g_2(\rho)\simeq \rho^{2H_0}$.
Hence, the following integral
\[
\int_1^\infty  \rho^{2d-2|H|-2H_0-3}  g_2(\rho)  d\rho
\]
is finite if  $2d-2|H|-2H_0-3<-1$, i.e.,    % and $ 2d-2|H|-2H_0-3+2H_0>-1$
\begin{equation}\label{cond1}
 |H|+H_0>d-1 \,.
\end{equation}

On the other hand, integration by parts yields
\begin{align}\label{g_1}
g_1(\rho)&= \int_0^{\rho}(\rho-  s)\cos(s)\cdot  s ^{2H_0-2}d  s \nonumber\\
&= \int_0^{\rho} s ^{2H_0-2}d   \left[\rho \sin (s) -s\sin (s)-\cos (s)+1\right] \nonumber\\
&=s ^{2H_0-2}    \left[\rho \sin (s) -s\sin (s)-\cos (s)+1\right] \big|_{s=0}^{\rho} \\
&\qquad -(2H_0-2) \int_0^{\rho}  \left[\rho \sin (s) -s\sin (s)-\cos (s)+1\right]  s ^{2H_0-3}  d  s\nonumber \,.
\end{align}
{%\red  
For sufficiently large $\rho$, the first term in \eqref{g_1} becomes negligible, so that 
\begin{align*}
	g_1(\rho)  \approx&     \int_0^{\rho}  \left[\rho \sin s -s\sin s-\cos s+1\right]  s ^{2H_0-3}  d  s \\
 =& \int_0^{\rho}  \left(\rho  -s\right)\cdot(\sin s )\cdot s ^{2H_0-3}  d  s +\int_0^{\rho}(1-\cos s)s ^{2H_0-3}  d  s\\
 \approx&\int_0^{\rho}  \left(\rho -s\right)\cdot(\sin s) \cdot s ^{2H_0-3}  d  s\,.
\end{align*}
Thus, a standard change of variable yields that
\begin{align}\label{g_1_infty}
 g_1(\rho)  
 &\approx \rho^{2H_0-1}\int_0^{1}  (1-s)\cdot \sin(\rho s) \cdot s ^{2H_0-3}  d  s \nonumber\\
 &\approx   \rho^{2H_0-1}\int_0^{1}   \sin(\rho s) \cdot s ^{2H_0-3}  d  s\,.
\end{align}
Denote $I(\rho, H_0) :=\int_0^{1}   \sin(\rho s) \cdot s^{2H_0-3}  ds$. In \ref{appenB.1}, we show that 
\begin{align}\label{eq:sin-1F2}
	I(\rho, H_0) \approx   \rho\cdot   {}_1F_2(H_0-\tfrac12; \tfrac32,  H_0+\tfrac12; -\tfrac{\rho^2}{4}),
\end{align}
where we have applied the generalized hypergeometric function $_1F_2(a_1;b_1,b_2;z)$ with parameters $a_1,\ b_1,\ b_2$. 
}
According to \cite[Eq. 16.11.8]{OF2010} with
\[
p=1, \quad q=2, \quad 
a_1=H_0-\tfrac12 , \quad b_1=\tfrac32, \quad b_2=H_0+\tfrac12\,,
\]
we have as $\rho\to +\infty$
\begin{eqnarray*}
{}_1F_2(a_1; b_1,  b_2; -\tfrac{\rho^2}{4})
&\approx &\tfrac{\Gamma(b_1)\Gamma(b_2)}{\Gamma(a_1)} [ H_{1,2}(\tfrac{\rho^2}{4})+E_{1,2}(\tfrac{\rho^2}{4}e^{\pi i} ) +E_{1,2}(\tfrac{\rho^2}{4}e^{-\pi i} )]
%\\
%&\approx & \rho^{-2a_1}+\rho^{\nu} \\
%&\approx & \rho ^{-2H_0+1}
\end{eqnarray*}
where the functions $H_{1,2}$ and $E_{1,2}$ are borrowed from  \cite[16.11.1 and 16.11.2]{OF2010} with two more parameters $\kappa=2,~\nu=-2$:
\begin{align*}
	H_{1,2}(z)&=\sum_{k=0}^{\infty}\frac{(-1)^k}{k!} \cdot\frac{\Gamma(a_1+k)}{\Gamma(b_1-a_1-k)\Gamma(b_2-a_1-k)}\cdot z^{-a_1-k},\\
	E_{1,2}(ze^{\pm\pi i})&=(2\pi)^{-1/2}\cdot 2^{3/2}\cdot e^{2z^{1/2}e^{\pm\frac{\pi}{2}i}}\sum_{k=0}^{\infty}\lt[2(ze^{\pm\pi i})^{1/2}\rt]^{-2-k}.
\end{align*}
{%\red  
Substituting $z = \tfrac{\rho^2}{4}$ into the above series, we observe that, as $\rho \to \infty$, only the term corresponding to $k = 0$ contributes. Consequently,
\begin{align}\label{eq:1F2_infty}
	{}_1F_2(a_1; b_1,  b_2; -\tfrac{\rho^2}{4})
	&\approx  \rho^{-2a_1}+\rho^{-2} \approx  \rho ^{-2H_0+1}\,, \quad\hbox{ as } \rho \to \infty\,.
\end{align}
Alternatively, we provide a direct verification of \eqref{eq:1F2_infty} using the Mellin-Barnes integral in \ref{appenB.2}.
}

Combining \eqref{g_1_infty}, \eqref{eq:sin-1F2} and \eqref{eq:1F2_infty} proves that
\begin{align}\label{exact_rho}
	g_1(\rho)\approx \rho^{2H_0-1}\cdot \rho\cdot \rho ^{-2H_0+1} \approx  \rho\,.
\end{align} 
Besides,  to further illustrate its asymptotics, we plot some graphs of $g_1(\rho)$ defined by \eqref{g_1} under different $H_0 $ in Figure \ref{image_g1}, which also shows that $g_1(\rho)\approx \rho$ when $1/2<H_0<1$. 
Thus,
 \[\int_1^\infty  \rho^{2d-2|H|-2H_0-3}  g_1(\rho)  d\rho<\infty\]
 if and only if
\begin{equation}\label{cond2}
2d-2|H|-2H_0-2<-1\Leftrightarrow |H|+H_0>d-1/2.
\end{equation}
This verifies \eqref{I_finite} and consequently proves Theorem \ref{t.iff} when $1/2<H_0<1$. 
The proof of all cases in \eqref{e.1.iff} is complete.
\end{proof}

\begin{figure}[htbp]
\begin{subfigure}{.3\textwidth}
  \centering
  \includegraphics[width=1\linewidth]{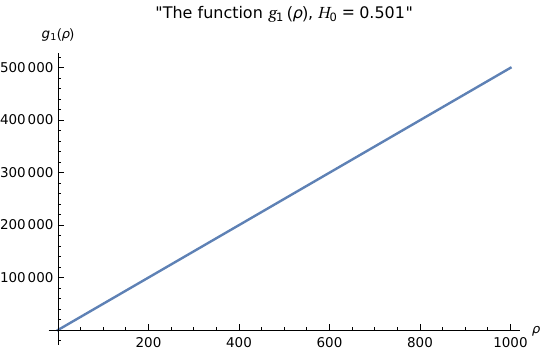}
 % \caption{$d_0=1$}
  \label{fig101}
\end{subfigure}%
 \hfill
\begin{subfigure}{.3\textwidth}
  \centering
  \includegraphics[width=1\linewidth]{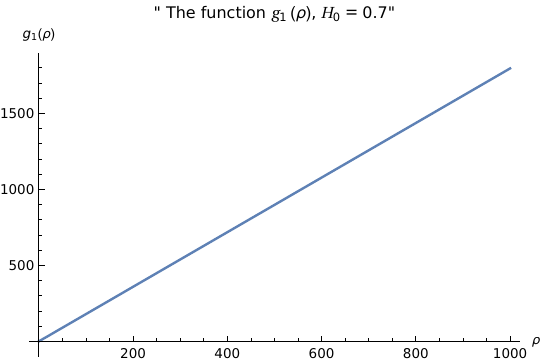}
  %\caption{$d_0=0.7$}
  \label{fig102}
\end{subfigure}
 \hfill
\begin{subfigure}{.3\textwidth}
  \centering
  \includegraphics[width=1\linewidth]{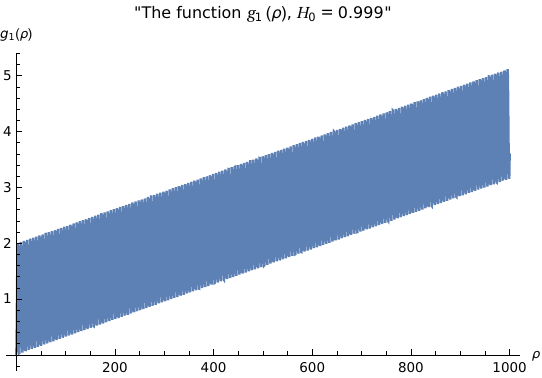}
  %\caption{$d_0=0.3$}
  \label{fig103}
\end{subfigure}
\caption{Images of $g_1(\rho)$ defined in \eqref{g_1} under three cases. \textbf{Left:} $H_0=0.501$, $0\leq\rho\leq1000$; \textbf{Middle:} $H_0=0.7$, $0\leq\rho\leq1000$; \textbf{Right:} $H_0=0.999$, $0\leq\rho\leq1000$.}
\label{image_g1}
\end{figure}

\begin{Rmk}\label{rem_trans}
As highlighted in the introduction, a discontinuity arises in the conditions \eqref{e.1.iff} of Theorem \ref{t.iff} when $H_0=1$. Examination of Figure \ref{image_g1} reveals that for $H_0\in(\tfrac12,1)$, the function $g_1(\rho)\asymp \rho$ shares the same order as $\rho$ as $\rho\to\infty$. Conversely, for $H_0=1$, the behavior of $g_1(\rho)$ deviates from being asymptotically of the order $\rho$ as $\rho\to\infty$ (refer to Figure \ref{rmk_image_g1}). This observation provides insight into the discontinuity in conditions \eqref{e.1.iff} at $H_0=1$.

\begin{figure}[htbp]
\begin{subfigure}{.45\textwidth}
  \centering
  \includegraphics[width=1\linewidth]{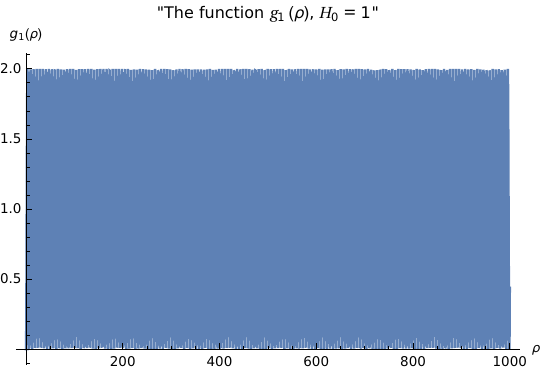}
 % \caption{$d_0=1$}
  \label{fig201}
\end{subfigure}%
 \hfill
\begin{subfigure}{.45\textwidth}
  \centering
  \includegraphics[width=1\linewidth]{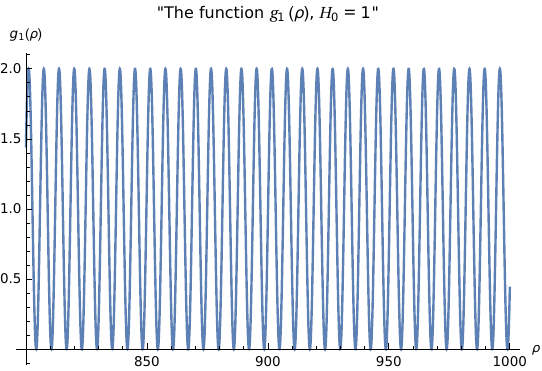}
  %\caption{$d_0=0.7$}
  \label{fig202}
\end{subfigure}
\caption{Images of $g_1(\rho)$ when $H_0=1$. \textbf{Left:} $H_0=1$, $0\leq\rho\leq1000$; \textbf{Right:} $H_0=1$, $800\leq\rho\leq1000$.}
\label{rmk_image_g1}
\end{figure}

\end{Rmk}

\section{Properties of the solution}

In this section, we focus on the properties of the solution to \eqref{eq.SWE} in one spatial dimension, particularly in the case where the noise is white in time ($H_0=\tfrac12$). The proofs of Theorem \ref{LUEESup}, Theorem \ref{LUHolderEESup}, and Theorem \ref{LUtHolderEESup} are based on Talagrand's majorizing measure theorem and the Sudakov minoration theorem. To this end, we begin by establishing sharp upper and lower bounds for the canonical metric associated with the solution $u(t,x)$:
\begin{align}\label{NaturalMetric}
   d_1((t,x),(s,y))= \sqrt{\EE|u (t,x)-u
 (s,y)|^2}\,.
\end{align}
It is important to note that $d_{1}((t,x), (s, y))$ does not represent a distance.

\begin{Lem}\label{LemMetricProp}
Let  $d_1((t,x),(s,y))$  be    the canonical metric  defined by  \eqref{NaturalMetric}. Denote 
\begin{align}\label{EqvMetric}
     D_{1,H}((t,x), (s, y)) &:= (s\wedge t)^{\frac 12}\cdot\big[|x-y|^H{\wedge  } (t\wedge s)^{H}\big]+(s \vee t)^{\frac12}\cdot|t-s|^{H}\,.
   \end{align}
Then,
\begin{equation}
 d_1((t,x),(s,y)) \approx D_{1,H}((t,x), (s, y))\,.
\end{equation}
This means that there exist  strict positive  constants $c_{H}$ and $C_{H}$ such that
    \begin{align}
   	c_H D_{1,H}((t,x), (s, y))\leq d_1((t,x),(s,y))
       \leq C_H D_{1,H}((t,x), (s, y)) \label{LUofNaturalMetric}
   \end{align}
 for all  $(t,x),(s,y)\in \RR_+\times\RR$.
\end{Lem}

\begin{Rmk}\label{Rmk:d_1} 
As shown in \cite[Lemma 3.6]{HW2021},  the canonical metric for the stochastic heat equation considered in their paper is approximated by
\[d_1((t,x),(s,y))\approx |x-y|^H\wedge(t\wedge s)^{\frac H2}+|t-s|^{\frac H2},\]
which differs from the expression in \eqref{EqvMetric}. This distinction leads to
% {\red diverse separation}  
a different size function $\Phi(T,L)$ in Theorem \ref{LUEESup}.

{%\red 
Previously, \cite{DalangBEJ10} characterized the canonical metric $d_1$ for the stochastic wave equation with Riesz noise associated with exponent $\beta$, establishing the local equivalence:
\begin{equation}\label{eq:DalangBEJ10_prop4.1}
d_{1}\bigl((t,x),(s,y)\bigr)
\approx \Bigl(|t-s|+\sum_{i=1}^{d} |x_{i}-y_{i}|\Bigr)^{1-\beta/2},
\end{equation}
for all $(t,x),(s,y)\in [a,a']\times [-b,b]^{d}$, where $0<a<a'<\infty$ and $0<b<\infty$.
In the one-dimensional setting with spatially correlated noise--i.e., when $H=1-\tfrac{\beta}{2}>\tfrac{1}{2}$--our estimate in \eqref{LUofNaturalMetric} recovers \eqref{eq:DalangBEJ10_prop4.1} on the domain $[a,a']\times [-b,b]$. 
It is possible, although technically demanding, that the global estimate \eqref{LUofNaturalMetric} for the canonical metric $d_{1}$ may be extended to higher dimensions. In particular, establishing an optimal lower bound would likely require constructing an analogue of the domain appearing in \eqref{E(xi)} in higher-dimensional space, which is challenging.
}
\end{Rmk}

\begin{proof}[Proof of Lemma \ref{LemMetricProp}]
%\tcb
{Without loss of generality, we assume $t>s$. By the isometry property stated in Proposition \ref{Int_Gen}, we have
\begin{align}\label{Re_d1}
    d_1^2((t,x), (s,y)) & =\EE[|u (t,x)-u(s,y)|^2] \nonumber\\
     &  = \EE\lt| \int_{0}^{s}\int_{\RR}[G_{t-r}(x-z)-G_{s-r}(y-z)] W(dr,dz) \rt|^2  \nonumber\\
    &\qquad  + \EE\lt| \int_{s}^{t}\int_{\RR}G_{t-r}(x-z) W(dr,dz) \rt|^2 \nonumber\\
    &=: d_{1,1}^2((t,x), (s,y))+ d_{1,2}^2((t,x), (s,y)).
    \end{align}}
{%\blue  
Thus, for $H\in(0,1)$ we have
\begin{align}\label{Re_d11}
   & d_{1,1}^2((t,x), (s,y)) \nonumber\\
   &\approx \int_0^s\int_{\RR} \bigg(\sin^2(|\xi| (t-r)) -2\sin(|\xi| (t-r))\cdot\sin(|\xi| (s-r)) \nonumber\\
   &\qquad\qquad\quad \cdot \cos(|\xi| |x-y|)
   +\sin^2(|\xi| (s-r))\bigg)\cdot|\xi|^{-1-2H}d\xi dr \nonumber\\
   &=:\mathfrak{I}((t,x),(s,y))\,,
\end{align}
and
\begin{align}\label{Re_d12}
    d_{1,2}^2((t,x), (s,y))&\approx\int_0^{t-s}\int_{\RR}|\xi|^{-1-2H}(\sin (r|\xi|))^2 d\xi dr\nonumber\\
 &=\int_0^{t-s}r^{2H}dr\cdot \int_{\RR}|\xi|^{-1-2H}(\sin |\xi|)^2 d\xi \nonumber\\
 &= C_H (t-s)^{2H+1}\,,
\end{align}
where $C_H$ is a constant depending only on $H$. Substituting \eqref{Re_d11} and \eqref{Re_d12} into \eqref{Re_d1}, we have
\begin{align}
    d_1^2((t,x), (s,y))
    & =:\mathfrak{I}((t,x),(s,y))+C_H (t-s)^{2H+1}.\label{d_1.J}
%=:&\mathfrak{I}+\frac{1}{2H+1} (t-s)^{2H+1}\,.\nonumber
\end{align}
}
It is straightforward to see 
\begin{align}\label{a1}
	\int_0^s\sin^2(|\xi| (t-r))dr=&\frac12 \int_0^s[1-\cos(2|\xi| (t-r))]dr \nonumber\\
	=&\frac{s}{2}+\frac{1}{4|\xi|}[\sin (2|\xi|(t-s))-\sin(2|\xi| t)]\,,
\end{align}
and
\begin{equation}\label{a2}
\int_0^s\sin^2(|\xi| (s-r))dr=\frac{s}{2}-\frac{1}{4|\xi|}\sin(2|\xi| s).
\end{equation}
Moreover, by a change of variable $s-r\rightarrow r$, we have
\begin{align}\label{a3}
\int_0^s&\sin(|\xi| (t-r))\cdot\sin(|\xi| (s-r))dr=\int_0^s\sin(|\xi| (t-s+r))\cdot\sin(|\xi| r)dr\nonumber\\
&=\int_0^s\big[\sin(|\xi| (t-s))\cdot \sin(|\xi| r)\cos(|\xi| r)+\sin^2(|\xi| r)\cdot \cos(|\xi| (t-s))\big]dr\nonumber\\
&=\frac12\sin(|\xi| (t-s))\cdot\int_0^s \sin(2|\xi| r)dr\nonumber\\
&\qquad\qquad +\frac12\cos(|\xi| (t-s))\cdot\int_0^s[1-\cos(2|\xi| r)]dr\nonumber\\
%&=\frac{1}{4|\xi|}\sin(|\xi| (t-s))\cdot [1-\cos(2|\xi| s)]+\frac{s}{2}\cos(|\xi| (t-s))-\frac{1}{4|\xi|}\cos(|\xi| (t-s))\cdot \sin(2|\xi| s)\\
&=-\frac{1}{4|\xi|}\sin(|\xi| (t+s))+\frac{1}{4|\xi|}\sin(|\xi| (t-s))+\frac{s}{2}\cos(|\xi| (t-s))\,.
\end{align}
Combining   \eqref{a1}, \eqref{a2} and \eqref{a3},   the quantity   $\mathfrak{I}((t,x),(s,y))$ defined in \eqref{Re_d11} can be simplified as
\begin{align}
\mathfrak{I}&((t,x),(s,y))\nonumber\\
&\approx s\cdot \int_{\RR^+}|\xi|^{-1-2H}\cdot\bigg(1-\cos[(t-s)|\xi|]\cdot \cos[|x-y||\xi|]\bigg)d\xi \nonumber \\
&\quad+\int_{\RR^+}|\xi|^{-2-2H}\cdot\bigg(\sin[2(t-s)|\xi|]-2\sin[(t-s)|\xi|]\cdot\cos[|x-y||\xi|] \nonumber \\
&\qquad\qquad+2\sin[(t+s)|\xi|]\cdot\cos[|x-y||\xi|]-\sin(2t|\xi|)-\sin(2s|\xi|)\bigg)d\xi \nonumber \\
&=: \fI_1((t,x),(s,y))+\fI_2((t,x),(s,y))\,.
%=:\fI_1+\fI_2 \,.
\label{Def.J1J2}
\end{align}
For simplicity, we denote $\fI_1:=\fI_1((t,x),(s,y))$ and $\fI_2:=\fI_2((t,x),(s,y))$. This is, we decompose the canonical metric into
\begin{equation}\label{NMPlancherel}
	d_1^2((t,x),(s,y))=\fI_1+\fI_2+C_H (t-s)^{2H+1}\,.
\end{equation}
We shall treat the lower   and upper bound parts of $d_1((t,x),(s,y))$ %in \eqref{LUofNaturalMetric} 
separately. Let us
first focus on the upper bound part.

\smallskip
\noindent\textbf{Step 1: The upper bound of \eqref{LUofNaturalMetric}}.
The triangle  inequality yields
%$\footnote{This might be too rough for the wave equation case.}$
	\begin{equation}\label{UConstant}
     d_1 ((t,x), (s,y))\le   d_1 ((t,x), (s,x))+  d_1 ((s,x), (s,y))\,.
	\end{equation}
Next,  we will deal with the above two terms by dividing them into 
 several more terms. 	
	
We first consider $d_1 ((t,x), (s,x))$. By \eqref{NMPlancherel},
\[
d_1^2 ((t,x), (s,x))=\fL_1+\fL_2+C_{H}(t-s)^{2H+1},
\]
where $\fL_1= \fL_1(t,s,x):=\fI_1((t,x),(s,x))$ and $\fL_2= \fL_2(t,s,x):=\fI_2((t,x),(s,x))$.
%Namely,
%\[\fL_1:=s\int_{\RR^+}|\xi|^{-1-2H}\cdot[1-\cos((t-s)|\xi|)]d\xi,\]
%and
%\begin{align*}
%\fL_2:&=\int_{\RR^+}|\xi|^{-2-2H}\bigg(2\sin[(t+s)|\xi|]-2\sin[(t-s)|\xi|]\\
%&\qquad\qquad+\sin[2(t-s)|\xi|]-\sin(2t|\xi|)-\sin(2s|\xi|)\bigg)d\xi.
%\end{align*}
By changing of variable $\xi (t-s)\rightarrow \xi$ we have
\begin{align*}
\fL_1&=s\cdot |t-s|^{2H}\int_{\RR^+}|\xi|^{-1-2H}\cdot[1-\cos(|\xi|)]d\xi\\
&\approx s\cdot |t-s|^{2H}.
\end{align*}
For the term $\fL_2$, we have
\begin{align*}
\fL_2=&\int_{\RR^+}|\xi|^{-2-2H}\cdot[\sin(2(t-s)|\xi|)-2\sin((t-s)|\xi|)]d\xi\\
&+\int_{\RR^+}|\xi|^{-2-2H}\cdot[2\sin((t+s)|\xi|)-\sin(2 t|\xi|)-\sin(2s|\xi|)]d\xi\\
=:&\fL_{21}+\fL_{22}.
\end{align*}
The change of variable $(t-s)|\xi|\rightarrow\xi$ gives 
\[\fL_{21}\approx |t-s|^{2H+1}.\]
Notice that
\begin{align*}
\fL_{22}&=\frac{-1}{1+2H}\int_{\RR^+}[2\sin((t+s)|\xi|)-\sin(2t|\xi|)-\sin(2s|\xi|)]d|\xi|^{-1-2H}\\
&=\frac{2}{2H+1}\int_{\RR^+}|\xi|^{-1-2H}[(t+s)\cos((t+s)|\xi|)-t\cos(2t|\xi|)-s\cos(2s|\xi|)]d\xi\\
&=\frac{2}{2H+1}\int_{\RR^+}|\xi|^{-1-2H}\bigg(t[1-\cos(2t|\xi|)]+s[1-\cos(2s|\xi|)]\\
&\qquad\qquad\qquad\qquad\qquad\qquad\qquad\qquad\quad-(t+s)[1-\cos(2(t+s)|\xi|)]\bigg)d\xi\\
&= C_H[t^{2H+1}+s^{2H+1}-(t+s)^{2H+1}]\leq0.
\end{align*}
Thus, we have
\[
d_1^2 ((t,x), (s,x))\leq s\cdot |t-s|^{2H}+C_H |t-s|^{2H+1},
\]
which means
\begin{equation}\label{d_11}
d_1 ((t,x), (s,x))\leq C_H\lt(s^{1/2}\cdot |t-s|^{H}+|t-s|^{H+\frac12}\rt).
\end{equation}
This gives the upper bound of $d_1 ((t,x), (s,x))$.

Now let us deal with $d_1 ((s,x), (s,y))$. An application of triangle   inequality  yields
\[d_1^2 ((s,x), (s,y))=\widetilde{\fL}_1+\widetilde{\fL}_2,\]
where $\widetilde{\fL}_1=\widetilde{\fL}_1(s,x,y):=\fI_1((s,x),(s,y))$ and $\widetilde{\fL}_2=\widetilde{\fL}_1(s,x,y):=\fI_2((s,x),(s,y))$.
%Namely,
%\[\widetilde{\fL}_1:=s\int_{\RR^+}|\xi|^{-1-2H}\cdot[1-\cos(|x-y||\xi|)]d\xi\]
%and \[
%\widetilde{\fL}_2:=2\int_{\RR^+}|\xi|^{-2-2H}\cdot \sin(2s|\xi|)\cdot[\cos(|\xi|(x-y))-1]d\xi.
%\]
By changing of variable $\xi |x-y|\rightarrow \xi$, it is easy to obtain
\[
\widetilde{\fL}_1\approx s|x-y|^{2H}.
\]
Since $\sin(2s|\xi|)\leq 2s|\xi| $ for $|\xi|\geq 0$, we have
\begin{align*}
\widetilde{\fL}_2\leq4s\int_{\RR^+}|\xi|^{-1-2H}\cdot[1-\cos(|\xi||x-y|)]d\xi\approx s|x-y|^{2H}.
\end{align*}
On the other hand,
\begin{align*}
      d_1^2((s,x),(s,y))
      =&\EE[|u (s,x)-u (s,y)|^2]\\
      \leq& 2\blc\EE[|u (s,x)|^2]+\EE[|u (s,y)|^2]\brc \leq C_H s^{2H+1}\,.
\end{align*}
Thus, it holds that
\begin{equation}\label{d_12}
d_1((s,x),(s,y))\leq C_H s^{\frac12}\cdot (|x-y|^H\wedge s^H).
\end{equation}
This shows the upper bound of $d_1((s,x),(s,y))$.

Combining \eqref{d_11} and \eqref{d_12}, we obtain the upper bound of $d_1((s,x),(t,y))$ as follows: 
\begin{align}\label{upper_d1}
d_1((s,x),(t,y))\leq&d_1((s,x),(t,y))+d_1((s,x),(s,y))\nonumber\\
\leq & C_H (s\wedge t)^{\frac12}\cdot|t-s|^{H} \nonumber\\
&+ C_{H}(t\wedge s)^{\frac 12}\cdot[|x-y|^H\wedge (t\wedge s)^{H}]+C_{H}|t-s|^{H+\frac 12}.
\end{align}
%So we finish the proof of the upper bound part of \eqref{LUofNaturalMetric}.

\smallskip
\noindent\textbf{Step 2: The lower bound of \eqref{LUofNaturalMetric}}.
Now we focus on establishing the lower bound for $d_1((t,x),(s,y))$.  We will divide the proof into two cases based on the value of $|x-y|$: $|x-y|>\alpha_H s$ and  $|x-y|\leq\alpha_H s$, where $\alpha_H>0$ is a constant to be determined later.
In particular, when  $|x-y|\leq\alpha_H s$, we further subdivide into another two cases: $(t-s)>\be_H s $ and $(t-s)<\be_H s $, for some constant $\be_H>0$.

\medskip
\noindent\emph{\textbf{Case 1:} $|x-y|>\alpha_H s$}.  Define the intervals \[[a_n,b_n]:=\lt[\frac{\pi}{3}+2n\pi,\frac{\pi}{2}+2n\pi\rt]\]
for nonnegative integer $n\in \mathbb{Z}_+$, and set
\begin{align}\label{E(xi)}
%E(&\xi):=\lt\{\xi\in\RR_+: s|\xi|>R \hbox{ and } 0<\cos(|x-y||\xi|)<\frac12 \rt\}\\
E(\xi):=\lt\{\xi\in\RR_+:  s|\xi|>\frac{R}{\alpha_H} \hbox{ and } |x-y||\xi| \in\bigcup_{n=0}^{\infty}[a_n,b_n]\rt\}\,,
\end{align}
where $R>0$ is a sufficiently large  (but fixed) constant.

It is straightforward to verify that the integrand inside $\int_{0}^t\int_{\RR} d\xi dr$ of \eqref{d_1.J} is non-negative. Therefore, we have
\begin{align}\label{fI_prime}
\fI&((t,x),(s,y))=\fI_1+\fI_2\nonumber\\
&\geq s\cdot \int_{E(\xi)}|\xi|^{-1-2H}\cdot\bigg(1-\cos[{(t-s)}|\xi|]\cdot \cos[|x-y||\xi|]\bigg)d\xi\nonumber\\
&\quad+\int_{E(\xi)}|\xi|^{-2-2H}\cdot\bigg(2\sin[(t+s)|\xi|]\cdot\cos[|x-y||\xi|]+\sin[2(t-s)|\xi|]\nonumber\\
&\qquad\qquad-2\sin[(t-s)|\xi|]\cdot\cos[|x-y||\xi|]-\sin(2t|\xi|)-\sin(2s|\xi|)\bigg)d\xi\nonumber\\
&=: \fI'_1+\fI'_2 \,.
\end{align}
For the term $\fI'_1$, by changing of variable $\widehat{\xi}=|x-y|\xi$ and noticing that $0<\cos(|x-y||\xi|)<\frac12$ on the set $E(\xi)$  we have
\begin{align}\label{est_I_1}
%\fI'_1&=s\int_{E(\xi)}|\xi|^{-1-2H}\big\{[1-\cos(|x-y|\xi)]+\cos(|x-y|\xi)[1-\cos(|t-s|\xi)]\big\}d\xi\nonumber\\
\fI'_1&\geq s\int_{E(\xi)}|\xi|^{-1-2H}[1-\cos(|x-y||\xi|)]d\xi\nonumber\\
&\geq \frac{s}{2}|x-y|^{2H}\int_{E(\widehat{\xi})}|\widehat{\xi}|^{-1-2H}d\widehat{\xi},
\end{align}
where the set $E(\widehat{\xi})$ is given by
\[E(\widehat{\xi}):=\lt\{\widehat{\xi}\in\RR_+:  |\widehat{\xi}|>\frac{|x-y|}{\alpha_Hs}R \hbox{ and }|\widehat{\xi}|\in\bigcup_{n=0}^{\infty}[a_n,b_n]\rt\}.\]
Define
\begin{equation}\label{N(an)}
N:=\inf\lt\{n: a_n>\frac{|x-y|}{\alpha_H s}R\rt\}.
\end{equation}
We know $a_N>\frac{|x-y|}{\alpha_H s}R>a_{N-1}=a_N-2\pi$, which means $a_N<\frac{|x-y|}{\alpha_H s}R+2\pi.$ Thus,
\begin{align}\label{est_intxi}
\int_{E(\widehat{\xi})} |\widehat{\xi}|^{-1-2H}d\widehat{\xi} &\geq \sum_{n=N}^{\infty}\int_{a_n}^{b_n}|\widehat{\xi}|^{-1-2H}d\widehat{\xi}\nonumber\\
&\geq \frac{1}{12}\sum_{n=N}^{\infty}\int_{a_n}^{a_n+2\pi}|\widehat{\xi}|^{-1-2H}d\widehat{\xi}\nonumber\\
&=\frac{1}{12}\int_{a_N}^{\infty}|\widehat{\xi}|^{-1-2H}d\widehat{\xi}\nonumber\\
&= c_H|a_N|^{-2H}\geq c_H\lt|\frac{|x-y|}{\alpha_H s}R+2\pi\rt|^{-2H}.
\end{align}
Combining \eqref{est_I_1} and \eqref{est_intxi} yields
\begin{align}\label{est_fI_pr}
\fI'_1&\geq c_H s|x-y|^{2H}\frac{\alpha_H^{2H}s^{2H}}{(|x-y|R+2\pi s\alpha_H)^{2H}}\nonumber\\
&= c_H s^{2H+1}\frac{\alpha_H^{2H}}{(R+\frac{2\pi s\alpha_H}{|x-y|})^{2H}}\nonumber\\
&\geq s^{2H+1}\frac{\alpha_H^{2H}}{(R+2\pi)^{2H}},
\end{align}
where  we used  the key assumption $|x-y|>\alpha_H s$ in the last inequality above.

Using the elementary identity $\sin(2t|\xi|)+\sin(2s|\xi|)=2\sin((t+s)|\xi|)\cos((t-s)|\xi|),$
the term $\fI'_2 $    can be simplified as
\begin{align}\label{I_2_pr}
%\fI'_2=&\int_{E(\xi)}|\xi|^{-2-2H}\bigg\{[\sin(2(t-s)|\xi|)-2\sin((t-s)|\xi|)\cos(|x-y||\xi|)]\\
%&\qquad\qquad\qquad+2\sin((t+s)|\xi|)\cdot[\cos(|x-y||\xi|)-\cos(|t-s||\xi|)]\bigg\}d\xi\\
\fI'_2=&\int_{E(\xi)}|\xi|^{-2-2H}[2\sin((t+s)|\xi|)-2\sin((t-s)|\xi|)] \nonumber\\
&\qquad\qquad\qquad\qquad\qquad\cdot[\cos(|x-y||\xi|)-\cos(|t-s||\xi|)]d\xi\nonumber\\
=&4\int_{E(\xi)}|\xi|^{-2-2H}\sin(s|\xi|)\cos(t|\xi|)\cdot\lt[\cos(|x-y||\xi|)-\cos(|t-s||\xi|)\rt]d\xi.
\end{align}
As a consequence, we see
\begin{align}\label{abs_fI_2}
|\fI'_2|&\ls \int_{E(\xi)}|\xi|^{-2-2H}|\sin(s|\xi|)|d\xi \ls \int_{E(\xi)}|\xi|^{-2-2H}\frac{\alpha_H s|\xi|}{R}d\xi\nonumber\\
&\ls \frac{s}{R}\int_{\{\xi\in \RR_+:s|\xi|>\frac{R}{\alpha_H}\}}|\xi|^{-1-2H}d\xi \ls \frac{s^{2H+1}\alpha_H^{2H+1}}{R^{2H+1}}.
\end{align}
Therefore, for the case $|x-y|>\alpha_H s,$  from \eqref{est_fI_pr} and \eqref{abs_fI_2} we have
\begin{align}\label{Est_I_case1}
\fI((t,x),(s,y))&=\fI_1+\fI_2\geq\fI'_1+\fI'_2\geq \fI'_1-|\fI'_2|\nonumber\\
&\geq c_{1,H}\frac{s^{2H+1}\alpha_H^{2H}}{(R+2\pi)^{2H}}-c_{2,H}\frac{s^{2H+1}\alpha_H^{2H+1}}{R^{2H+1}}\nonumber\\
&\geq c_H\cdot  s^{2H+1}  \geq c_H\cdot s\lt(|x-y|\wedge  s\rt)^{2H} ,
\end{align}
when $R$ is sufficiently large.

This is, we prove that for any $t>s$,  there exist a positive $c_H>0$ such that
\begin{equation}\label{est_case1}
\fI((t,x),(s,y))\geq c_H\cdot s\lt(|x-y|\wedge  s\rt)^{2H}.
\end{equation}

\medskip

\noindent\emph{\textbf{Case 2:} $|x-y|\leq\alpha_Hs$}. Recall that $\fI_1$ and $\fI_2$ are defined by \eqref{Def.J1J2}. By the trivial identities
\begin{align*}
	1-\cos&(|t-s||\xi|)\cdot \cos(|x-y||\xi|) \\
	=&[1-\cos(|t-s||\xi|)]+\cos(|t-s||\xi|)[1-\cos(|x-y||\xi|)] \\
	=&[1-\cos(|x-y||\xi|)]+\cos(|x-y||\xi|)[1-\cos(|t-s||\xi|)] \,,
\end{align*}
and by a change   of variable $(t-s)\xi\rightarrow \xi,$ we have
\begin{align}\label{I_1_lower}
%\fI_1&\geq s\cdot \int_{D_+}|\xi|^{-1-2H}\cdot\bigg(1-\cos[(t-s)|\xi|]\cdot \cos[(x-y)|\xi|]\bigg)d\xi\nonumber\\
\fI_1&=s \int_{\RR_+}|\xi|^{-1-2H}\big\{\lt(1-\cos(|t-s||\xi|)\rt)\nonumber\\
&\qquad\qquad\qquad\qquad\quad+\cos(|t-s||\xi|)\lt(1-\cos(|x-y||\xi|)\rt)\big\}d\xi\nonumber\\
&=s|t-s|^{2H} \int_{\RR_+}|\xi|^{-1-2H}\bigg\{(1-\cos|\xi|)\nonumber\\
&\qquad\qquad\qquad\qquad\quad+(\cos|\xi|)\lt[1-\cos\lt(\frac{|x-y|}{|t-s|}|\xi|\rt)\rt]\bigg\}d\xi\nonumber\\
&\geq s|t-s|^{2H}\int_{D_+}|\xi|^{-1-2H}(1-\cos|\xi|)d\xi\nonumber\\
&\gs s|t-s|^{2H},
\end{align}
with $D_+:=\{\xi:\cos\xi>0\}$.
Similarly, we can obtain that
\begin{equation}\label{fI_1x,y}
\fI_1\geq c_H s|x-y|^{2H}.
\end{equation}
%Thus, we have
%\begin{equation}\label{est_fI_1}
%\fI_1\gs s\lt[|t-s|^{2H}+|x-y|^{2H}\rt].
%\end{equation}

As for the term $\fI_2$ in this case, using the  same  way  as  that for  $\fI'_2$ (e.g.  \eqref{I_2_pr}), we have
\begin{align*}
\fI_2&=4\int_{\RR_+}|\xi|^{-2-2H}\sin(s|\xi|)\cos(t|\xi|)\cdot\lt[\cos(|x-y||\xi|)-\cos(|t-s||\xi|)\rt]d\xi\\
&=4\int_{\RR_+}|\xi|^{-2-2H}\sin(s|\xi|)\cos(t|\xi|)\cdot\lt[\cos(|x-y||\xi|)-1\rt]d\xi\\
&\qquad+4\int_{\RR_+}|\xi|^{-2-2H}\sin(s|\xi|)\cos(t|\xi|)\cdot\lt[1-\cos(|t-s||\xi|)\rt]d\xi\\
&=:\fI_{21}+\fI_{22}.
\end{align*}
It is not difficult to obtain that
\begin{align*}
|\fI_{21}|&\ls\int_{\RR_+}|\xi|^{-2-2H}\lt[1-\cos(|x-y||\xi|)\rt]d\xi\\
&\leq C_H|x-y|^{2H+1}.
\end{align*}
Similarly, we have
\begin{align*}
|\fI_{22}|\leq C_H|t-s|^{2H+1}.
\end{align*}
Thus, when $|x-y|<\alpha_H s,$ by \eqref{fI_1x,y} we have
\begin{align*}
\frac12\fI_1+\fI_{21}&\geq \frac12\fI_1-|\fI_{21}|\\
&\geq c_{1,H}\cdot s|x-y|^{2H}-c_{2,H}\cdot|x-y|^{2H+1}\\
&\geq c_{1,H}\cdot s|x-y|^{2H}-c_{2,H}\cdot\alpha_H s|x-y|^{2H}\\
&\geq c_{3,H}\cdot s|x-y|^{2H}=c_{3,H}\cdot s(|x-y|\wedge   s)^{2H},
\end{align*}
with $c_{3,H}>0$ provided that we choose $\alpha_H$ such that $0<\al_H<\lt(\frac{c_{1,H}}{c_{2,H}}\wedge 1\rt)$.

Therefore, we show that when $|x-y|<\alpha_Hs$, it holds
\begin{equation}\label{not_im1}
\frac12\fI_1+\fI_{21}\geq c_{3,H}\cdot s(|x-y|\wedge   s)^{2H}.
\end{equation}
To obtain the lower bound of $\fI((t,x),(s,y))$, we have to find the lower bound of $(\frac12\fI_1+\fI_{22})$.   We shall show this by considering the following two cases: $ (t-s)<\be_H s $ and $ (t-s)\geq \be_H s $ for some constant $\be_H>0$.
\medskip

%\begin{align}\label{est_fI_case31}
%\fI((t,x),(s,y))&=\fI_1+\fI_2\geq\frac12 \fI_1+\fI_{21}\nonumber\\
%&\geq c_{H}\cdot s(|x-y|\wedge   s)^{2H}+c_{H}\cdot s(|t-s|\wedge   s)^{2H}.
%\end{align}

\noindent\emph{The case $|x-y|<\alpha_H s$ and $|t-s|<\be_H s$}. In the same manner,  from \eqref{I_1_lower} we see that
\begin{align*}
 \frac12\fI_1+\fI_{22}&\geq c_{H}\cdot s|t-s|^{2H}=c_{H}\cdot s(|t-s|\wedge   s)^{2H}.
\end{align*}
Consequently, from \eqref{not_im1} and the above inequality,
\begin{align}\label{est_fI_case3}
\fI((t,x),(s,y))&=\fI_1+\fI_2=\frac12 \fI_1+\fI_{21}+\frac12 \fI_1+\fI_{22}\nonumber\\
&\geq c_{H}\cdot s(|x-y|\wedge   s)^{2H}+c_{H}\cdot s(|t-s|\wedge   s)^{2H}.
\end{align}
\medskip
\noindent\emph{The  case $|x-y|<\alpha_H s$ and $(t-s)\geq\be_H s $ }.   Let us denote
\begin{align}\label{F(xi)}
%E(&\xi):=\lt\{\xi\in\RR_+: s|\xi|>R \hbox{ and } 0<\cos(|x-y||\xi|)<\frac12 \rt\}\\
F(\xi):=\lt\{\xi\in\RR_+:  s|\xi|>\frac{R}{\be_H} \hbox{ and } (t-s)|\xi| \in\bigcup_{n=0}^{\infty}[a_n,b_n]\rt\}\nonumber\,,
\end{align}
and
\[
\widetilde{N}:=\inf\lt\{n: a_n>\frac{|x-y|}{\be_H s}R\rt\}.
\]
This is, $0<\cos[(t-s)|\xi|]<\frac12$ on the set $F(\xi)$. Similarly to the case $|x-y|>\alpha_Hs$, we can prove that
\begin{align}\label{Est_I_case2}
\fI((t,x),(s,y))\geq c_H\cdot s\lt(|t-s|\wedge  s\rt)^{2H}.
\end{align}
Combining \eqref{est_case1}, \eqref{est_fI_case3} and \eqref{Est_I_case2} yields  that when $t>s$
\[\fI((t,x),(s,y))\geq c_{H}\cdot s(|x-y|\wedge   s)^{2H}+c_{H}\cdot s(|t-s|\wedge   s)^{2H},\]
which indicates from \eqref{NMPlancherel} that
\begin{align}\label{lower_d1}
d_1((t,x),(s,y))&\geq c_{H}\cdot (s\wedge t)^{\frac12}\cdot\big[|x-y|\wedge   (s\wedge t)\big]^{H}\nonumber\\
&\quad+c_{H}\cdot (s\wedge t)^{\frac12}\cdot\big[|t-s|\wedge   (s\wedge t)\big]^{H}+c_{H}|t-s|^{H+\frac12}.
%&=: c_{H}\cdot (s\wedge t)^{\frac12}\cdot\big[|x-y|\wedge   (s\wedge t)\big]^{H}+\tilde{d_1}((t,x),(s,y))\,.
\end{align}
Denote \[\tilde{d_1}((t,x),(s,y):=c_{H}\cdot (s\wedge t)^{\frac12}\cdot\big[|t-s|\wedge   (s\wedge t)\big]^{H}+c_{H}|t-s|^{H+\frac12}.\]
 It is not difficult to verify that when $t>s$, $(t-s)\geq\be_H s $ and $\be_H>0$,
%the following inequality holds
%by discussing two cases:  (I) $s<t<(1+\be_H)s$  and (II) $t>(1+\be_H)s$.
%\footnote{Details can be seen form Remark \ref{3.2}.}
\begin{align}\label{2_est_d1}
\tilde{d_1}((t,x),(s,y))&=c_{H}s^{1/2}[(t-s)\wedge   s]^{H}+c_{H}(t-s)^{H+1/2}\nonumber\\
&\geq \lt(\frac{c_{H}\be_H}{2}\wedge c_{H}\rt)\cdot s^{1/2}(t-s)^{H}+\frac{c_{H}}{2}(t-s)^{H+1/2}\nonumber\\
&=c_{1,H}\cdot s^{1/2}(t-s)^{H}+c_{2,H}(t-s)^{H+1/2}.
\end{align}
 Accordingly, from \eqref{lower_d1} and \eqref{2_est_d1}, we obtain the lower bound
\begin{align}\label{Lower_d1}
d_1((t,x),(s,y))&\geq c_{H}\cdot (s\wedge t)^{\frac12}\cdot\big[|x-y|\wedge   (s\wedge t)\big]^{H}\nonumber\\
&\quad+c_{H}\cdot (s\wedge t)^{\frac12}\cdot|t-s|^{H}+c_{H}|t-s|^{H+\frac12}\,.
\end{align}

As a result, combining \eqref{upper_d1} and \eqref{Lower_d1}, we have
\begin{align*}
 d_1((t,x),(s,y))\approx
 (s\wedge t)^{\frac 12}\cdot[|x-y|^H\wedge (s\wedge t)^{H}]&+(s\wedge t)^{\frac12}\cdot|t-s|^{H}\\
 &+
|t-s|^{H+\frac 12}.
\end{align*}
Moreover, it is clear that
\[
 (s \vee t)^{\frac12} \approx (s\wedge t)^{\frac 12}+|t-s|^{\frac 12}\,.
\]
Thus, 
\begin{align*}
 d_1((t,x),(s,y)) &\approx (s\wedge t)^{\frac 12}\cdot\big[|x-y|^H{\wedge  } (s\wedge t)^{H}\big]+(s \vee t)^{\frac12}|t-s|^{H}.
 \end{align*}
Therefore,  the proof is complete.
\end{proof}

%\begin{Rmk}
%
%  The above property of  the canonical metric can also be written as
%   \begin{align}\label{EqvMetric}
%   	d_1((t,x),(s,y))\approx d_{1,H}((t,x), (s, y)):&=(t\wedge s)^{\frac 12}\cdot\big[|x-y|^H{\wedge  } (t\wedge s)^{H}\big] \nonumber\\
%      &\qquad + (s\wedge t)^{\frac12}\cdot|t-s|^{H}+|t-s|^{H+\frac 12}.
%   \end{align}
%
%\end{Rmk}

%\begin{Rmk}\label{3.2}
%	\footnote{can we delete this remark?}
%	We will see that \eqref{lower_d1} can be bounded below by the same form in \eqref{LUofNaturalMetric}. Assume $t>s$, and $\beta_H>0$. Clearly, we need to discuss two cases:  (I) $s<t<(1+\be_H)s$;  and (II) $t>(1+\be_H)s$. For the case (I), they are exactly same. So, we shall focus on case (II). However, it is not hard to see
%	\begin{align*}
%		c_{1,H} s[(t-s)\wedge s]^{2H}+c_{2,H} (t-s)^{2H+1}=&c_{1,H} s^{2H+1}+c_{2,H} (t-s)\cdot(t-%s)^{2H} \\
%		\geq& \frac{c_{2,H}\be_H}{2}s\cdot (t-s)^{2H}+\frac{c_{2,H}}{2} (t-s)^{2H+1}\,.
%	\end{align*}

%As a result, we can show that
%	\begin{align*}
%		c_{1,H} s[(t-s)\wedge s]^{2H}&+c_{2,H} (t-s)^{2H+1}\\
%\geq& \Blk c_{1,H}\wedge  \frac{c_{2,H}\be_H}{2}\Brk s(t-s)^{2H}+\frac{c_{2H}}{2} (t-s)^{2H+1} \\
%		=&c'_{1,H} s(t-s)^{2H}+c'_{2,H}(t-s)^{2H+1}\,.
%	\end{align*}
%\end{Rmk}

\begin{proof}[Proof of Theorem \ref{LUEESup}]
We follow a similar approach to the proof of \cite[Theorem 1.1]{HW2021}. However, the equivalent canonical metric in Lemma \ref{LemMetricProp} is more intricate than the one considered in \cite{HW2021}. We will highlight the differences and omit the details that are analogous.

\smallskip
\noindent\textbf{Step 1:} We show   the first part in Theorem \ref{LUEESup}. Still we shall use Talagrand's majorizing measure theorem (Theorem \ref{Talagrand}) to establish the upper bound and use the Sudakov minoration theorem (Theorem \ref{Sudakov}) to obtain the lower bound.
Let
\[
 \bT=[0,T]\quad{\rm and}\quad   \bL=\LL\,.
\]
Since { $u (t,x)$ is a symmetric and  centered Gaussian process, }  by Lemma \ref{equiv} in \ref{appenA}, we have
\begin{equation}\label{supE}
   \EE \lk\sup_{(t,x)\in\bT\times\bL} |u (t,x)|\rk\approx \EE \lk\sup_{(t,x)\in\bT\times\bL} u (t,x)\rk\,.
\end{equation}
Hence, to show \eqref{SupEEasmp} it is equivalent to showing
\begin{equation}
 { c_H\, \Phi(T, L)\le
\EE \lk\sup_{t\in \bT, x\in \bL}  u (t,x) \rk \le C_H\,  \Phi(T, L)\,,}
 \label{SupEEasmp.a}
    \end{equation}
 where $\Phi(T, L)$ is defined in \eqref{e.def_rho}, which differs from the corresponding quantity in \cite{HW2021}.

 For the upper bound in \eqref{SupEEasmp}, following an approach analogous to that in \cite[Theorem 1.1]{HW2021},  %\tcb
 {
 we choose the admissible sequences $(\mathcal{A}_n)$ as uniform partitions of $\bT\times \bL$ such that card$(\mathcal{A}_n)\leq 2^{2^n}$.} More precisely,  we partition $\bT\times \bL=[0,T]\times [-L,L]$  as follows:
 \begin{equation*}
 \begin{cases}  [0,T] =\bigcup\limits_{j=0}^{2^{2^{n-1}}-1}\lt[j\cdot2^{-2^{n-1}}T,(j+1)\cdot2^{-2^{n-1}}T\rt)\,, \\ \\
   [-L,L] =\bigcup\limits_{k=-2^{2^{n-2}}}^{2^{2^{n-2}}-1}\lt[k\cdot2^{-2^{n-2}}L,(k+1)\cdot2^{-2^{n-2}}L\rt)\,.
   \end{cases}
 \end{equation*}
 By Talagrand's majorizing measure theorem (Theorem \ref{Talagrand}), we have
  \begin{equation*}
  	\EE \lk\sup_{(t,x)\in\bT\times\bL} u (t,x)\rk\leq C\gamma_2(T,d)\leq C\sup_{(t,x)\in\bT\times\bL} \sum_{n\geq 0}2^{n/2} \diam(A_n(t,x))\,,
  \end{equation*}
where
 \[
  A_n(t,x)=\lt[j\cdot 2^{-2^{n-1}}T,(j+1)\cdot 2^{-2^{n-1}}T\rt)\times \lt[k\cdot 2^{-2^{n-2}}L,(k+1)\cdot 2^{-2^{n-2}}L\rt)
 \]
satisfies $j\cdot 2^{-2^{n-1}}T\leq t< (j+1)\cdot 2^{-2^{n-1}}T$ and $k\cdot 2^{-2^{n-2}}L \leq x<(k+1)\cdot 2^{-2^{n-2}}L$.
Now the diameter of $A_n(t,x)$ with respect to  $D_{1,H}((t,x), (s,y))$ defined by \eqref{EqvMetric} can be estimated as
 \[
   \diam(A_n(t,x))\leq C_{H}T^{\frac 12}\lc T^{H}\wedge (2^{-H 2^{n-2}}L^H) \rc+C_H 2^{-H2^{n-1}}T^{H+\frac 12}.
  \]
 Then by Theorem \ref{Talagrand} and  by dividing the discussion into  two cases $L>T$ and $L\leq T$, 
 we can prove the upper bound part of \eqref{SupEEasmp} in a similar way to that of
 \cite[Theorem 1.1]{HW2021}.

For the  lower bound part of \eqref{SupEEasmp},  choosing a sequence
$\{u(T,x_i),i=0,1,\cdots,\pm N\}$ when $L>T$,  where  
  \[
   x_0=0,x_{\pm1}=\pm T,\cdots, x_{\pm N}=\pm N T\,,
  \]
  and %\tcb
  {$N=\lfloor L/T \rfloor$ (note that $N \geq 1$ only when $L>T$)}. For this sequence,  we have
   \[
   D_{1,H}((T,x_i),(T,x_j))\geq c_H T^{\frac 12+H}=\delta  \quad\hbox{if $i\neq j$}\,.
   \]
When $L\leq T$, choosing sequence $\{u(T/2,0),u(T,0)\}$ holds that 
   \[
   D_{1,H}((T/2,0),(T,0))\geq c_H T^{\frac 12+H}=\delta  \,.
   \]
Then applying the Sudakov minoration theorem (Theorem \ref{Sudakov}), we obtain the following lower bound for $\EE [\sup_{(t,x)\in\bT\times\bL} u (t,x)]$:
 \[\EE \lk\sup\limits_{(t,x)\in\bT\times\bL} u (t,x)\rk\geq T^{\frac{1}{2}+H}\Phi_0(T, L).\]
 This proves the first part of  Theorem \ref{LUEESup}.

\smallskip
\noindent\textbf{Step 2:} Now we  prove the second part of Theorem \ref{LUEESup}. Denote \[\bL:=[-L,L], \quad \bT^{\alpha}=[0,n^{\alpha}].\]
We first prove \eqref{Supasmp} for $T=n^\al$ for some positive $\al$ to be determined and for sufficiently large $L\geq n^{(1+\varepsilon)\al}$. By the first part of Theorem \ref{LUEESup}, we have
	\[
	 \EE \lk\sup_{(t,x)\in\bT_n^{\alpha}\times\bL} u (t,x)\rk\geq c_H\left(
        n^{\al(H+\frac12)}+n^{\al(H+\frac12)}\sqrt{\log_2\lc\frac{L}{n^{\alpha}}\rc}\right)
	\]
    for some positive number $c_H$. Moreover,  it follows from direct computation that
    \[\sigma^2_{H}:=\sigma^2_{H}(\bT_n^\alpha\times\bL)=\sup\limits_{(t,x)\in\bT_n^\alpha\times\bL}\EE[|u (t,x)|^2]=C_H n^{\al(1+2H)}.\]
    Set $\lambda_{H}:=\lambda_{H}(\bT^\alpha\times\bL)=\frac{1}{2}\EE \lk\sup_{(t,x)\in\bT^\alpha\times\bL} u (t,x)\rk$.
Similar to the proof of  \cite[Theorem 1.1]{HW2021},  and  using Borell's inequality we have
\begin{equation}\label{BorellConsqL}
	\begin{split}
		\bP&\lt\{\sup_{(t,x)\in\bT^{\alpha}\times\bL} u (t,x)<\frac 12 \EE\Blk\sup_{(t,x)\in\bT^{\alpha}\times\bL}u
    (t,x)\Brk\rt\} \\
    & \leq2\exp\lc -\frac{\lambda^2_{H}}{2\sigma^2_{H}}\rc\leq C_H n^{-\alpha\varepsilon\cdot c_H}\,,
	\end{split}
	\end{equation}
	and
	 \begin{equation}\label{BorellConsqU}
    \begin{split}
     \bP&\lt\{\sup_{(t,x)\in\bT^{\alpha}\times\bL} u (t,x)>\frac 32 \EE\Blk\sup_{(t,x)\in\bT^{\alpha}\times\bL}u (t,x)\Brk\rt\}\\
    &\leq2\exp\lc -\frac{\lambda^2_{H}}{2\sigma^2_{H}}\rc
		\leq C_H n^{-\alpha\varepsilon\cdot c_H}\,.
    \end{split}
    \end{equation}
Then, by Borel-Cantelli's lemma with $\alpha>1/\varepsilon c_H$ and the property of the supremum function, we obtain the upper bound and lower bound parts of \eqref{Supasmp}. The proof is complete. 
\end{proof}

Now we give the proof of Theorem \ref{LUHolderEESup}.

\begin{proof}[Proof of Theorem \ref{LUHolderEESup}]
Recall the notation
\begin{equation}\label{Diff}
  \begin{split}
     \Delta _h u(t,x):=&u (t,x+h)-u(t,x)\\
       =&\int_{0}^{t}\int_{\RR} [G_{t-s}(x+h-z)-G_{t-s}(x-z)] W(ds,dz),
  \end{split}
  \end{equation}
  where $t>0$ and $h\neq 0$ are fixed . Without loss of generality, we assume $h>0$ in the following.
  Denote the associated %\tcb
  {canonical metric} as
    \begin{equation*}
    d_{2, t,h}(x,y):=\lc\EE| \Delta _h u(t,x)- \Delta _h u (t,y))|^2\rc^{\frac 12} \,.
  \end{equation*}
 %  Similarly  to \eqref{NMPlancherel},

 We first establish the upper bound in \eqref{SupHolderEEasmp}.  To this end, we need the upper bound of the metric $d_{2, t,h}(x,y)$. By applying Plancherel's identity with respect to $z$ in $d_{2, t,h}(x,y)$, and using the elementary inequality $1-\cos(x)\leq C_{\theta}x^{2\theta}$  for any  $\theta\in(0,H)$,
  \begin{align}\label{def_J2}
     d_{2, t,h}^2(x,y)&=C_H \int_{\RR_+}\lt(\frac t2 -\frac{\sin(2|\xi|t)}{4|\xi|}\rt)(1-\cos(|x-y|\xi))\nonumber\\
     &\qquad\qquad\qquad\qquad\qquad\qquad\qquad\quad\cdot(1-\cos(h\xi))\cdot\xi^{-1-2H}d\xi\nonumber\\
     &\leq  C_{H,\theta}h^{2\theta} \int_{\RR_+}\lt(t -\frac{\sin(2|\xi|t)}{2|\xi|}\rt)(1-\cos(|x-y|\xi))\cdot\xi^{2\theta-1-2H}d\xi\nonumber\\
     &=: C_{H,\theta}h^{2\theta}\cJ_2(t,x,y)\,,
  \end{align}
where  $\cJ_2(t,x,y)$ is equal to $\mathfrak{I}((t,x),(t,y))$ defined by \eqref{d_1.J}. Thus, analogous  to the proof of upper bound for $\mathfrak{I}((t,x),(t,y))$, with $H$ replaced  by $H-\theta$,   we obtain
\begin{equation*}
\cJ_2(t,x,y)\leq C_{H,\theta}t\cdot\lt(|x-y|^{2H-2\theta}\wedge  t^{2H-2\theta}\rt).
\end{equation*}
This gives
\begin{equation}\label{est_d2}
d_{2, t,h}(x,y)\leq C_{H,\theta}t^{\frac12}h^{\theta}\lt(|x-y|\wedge   t\rt)^{H-\theta}.
\end{equation}
As  in   the proof of Theorem \ref{LUEESup},   we take 
\[
  A_n(x)=\lt[k\cdot 2^{-2^{n-2}}L,(k+1)\cdot 2^{-2^{n-2}}L\rt)
 \]
and partition  $[-L,L]$ as 
\[[-L,L]=\bigcup\limits_{k=-2^{2^{n-2}}}^{2^{2^{n-2}}-1}A_n(x)=\bigcup\limits_{k=-2^{2^{n-2}}}^{2^{2^{n-2}}-1}\lt[k\cdot2^{-2^{n-2}}L,(k+1)\cdot2^{-2^{n-2}}L\rt).\]
The diameter of $A_n(x)$ with respect to $d_{2,t,h}(x,y)$ can be estimated as
\[\diam(A_n(x))\leq C_{H,\theta}t^{1/2}\cdot h^{\theta}\lt[\lt(2^{-2^{n-2}}L\rt)\wedge t\rt]^{H-\theta}.\]
Thus, by invoking Talagrand's majorizing measure theorem   (Theorem \ref{Talagrand}), the upper bound of $\EE [\sup_{x\in\bL} \Delta _h u(t,x)]$ can be obtained as follows: 
\begin{equation}\label{upper1}
  \EE \lk\sup_{x\in\bL} \Delta _h u(t,x)\rk\leq C_{ H,\theta} |h|^{\theta} t^{1/2+H-\theta} \Phi_0(t,L)\,.
  \end{equation}

Now we turn to prove the lower  bound  in \eqref{SupHolderEEasmp}.  To   this end, we  need to find the lower bound of $d_{2, t,h}(x,y)$. 
%Recall that $E(\xi)$ defined by \eqref{E(xi)} is
%\begin{align*}
%%E(&\xi):=\lt\{\xi\in\RR_+: s|\xi|>R \hbox{ and } 0<\cos(|x-y||\xi|)<\frac12 \rt\}\\
%\tcr{G(\xi):=\lt\{\xi\in\RR_+:  t|\xi|>\frac{R}{\alpha_H} \hbox{ and } |x-y||\xi| \in\bigcup_{n=0}^{\infty}\lt[\frac{\pi}{3}+2n\pi,\frac{\pi}{2}+2n\pi\rt]\rt\}}\,,
%\end{align*}
%and  $0<\cos[|\xi||x-y|]<1/2$ on the set $E(\xi)$. 
Since for fixed $0<c_0<\frac{\pi}{2h}$, $\sin(x)<x$ when $x\geq c_0$, we have
%\begin{align*}
%d_{2, t,h}^2(x,y)&=C_H \int_{\RR_+}\lt[\frac t2 -\frac{\sin(2|\xi|t)}{4|\xi|}\rt][1-\cos(|x-y|\xi)][1-\cos(h\xi)]\cdot\xi^{-1-2H}d\xi\nonumber\\
%&\geq t\int_{G(\xi)}(1-\cos (h|\xi|) )|\xi|^{-1-2H}\lt(1-\frac{\sin(2|\xi|t)}{2|\xi|t}\rt)d\xi\\
%%&=t \int_{E(\xi)}(1-\cos (h|\xi|) )|\xi|^{-1-2H}d\xi-\frac12 \int_{E(\xi)}(1-\cos (h|\xi|) )|\xi|^{-2-2H}\sin(2|\xi|t)d\xi\\
%&\geq c_Ht\int_{G(\xi)}(1-\cos (h|\xi|) )|\xi|^{-1-2H}d\xi\\
%&\geq c_H t h^{2H}\,.
%\end{align*}
\begin{align*}
d_{2, t,h}^2(x,y)
%&=C_H \int_{\RR_+}\lt[\frac t2 -\frac{\sin(2|\xi|t)}{4|\xi|}\rt][1-\cos(|x-y|\xi)][1-\cos(h\xi)]\cdot\xi^{-1-2H}d\xi\nonumber\\
&= \frac{C_H  t}{2}\int_{\RR_+}\lt(1-\frac{\sin(2|\xi|t)}{2|\xi|t}\rt)(1-\cos (h|\xi|))\\
&\qquad\qquad\qquad\qquad\qquad\cdot(1-\cos(|x-y|\xi))\cdot |\xi|^{-1-2H}d\xi\\
%&=t \int_{E(\xi)}(1-\cos (h|\xi|) )|\xi|^{-1-2H}d\xi-\frac12 \int_{E(\xi)}(1-\cos (h|\xi|) )|\xi|^{-2-2H}\sin(2|\xi|t)d\xi\\
&\gs ~ t\int_{c_0}^{\infty}(1-\cos (h|\xi|) )[1-\cos(|x-y|\xi)]\cdot |\xi|^{-1-2H}d\xi\\
&\approx t \cdot |x-y|^{2H}\int_{|x-y|c_0}^{\infty}\lt[1-\cos \lt(\frac{h\xi}{|x-y|} \rt)\rt][1-\cos(\xi)]\cdot |\xi|^{-1-2H}d\xi  \\
&\gs~ t h^2\cdot |x-y|^{2H-2}\int_{|x-y|c_0}^{\frac{|x-y|\pi}{2h}}[1-\cos(\xi)]\cdot |\xi|^{1-2H}d\xi\,,
\end{align*}
where in the last inequality we use the simple inequality $1-\cos(x)\geq x^2/4$ when $|x|\leq \pi/2$.   Set
  \[
   k_0=\inf \lt\{k\in\bN_0:\frac{(6k+1)\pi}{3}\geq |x-y|c_0\rt\}\,;
  \]
  \[
   k_1=\sup\lt\{k\in\bN_0:\frac{(6k+5)\pi}{3}\leq \frac{|x-y|\pi}{2h}\rt\}\,;
  \]
and % $I_k=(\frac{(2k+1)\pi}{2},\frac{(2k+3)\pi}{2}]$. 
$I_k=(\frac{(6k+1)\pi}{3},\frac{(6k+5)\pi}{3}]$. Then if $h$ is sufficiently small, we have
\begin{align}\label{est_devide}
\int_{|x-y|c_0}^{\frac{|x-y|\pi}{2h}} &[1-\cos(\xi)]\cdot |\xi|^{1-2H}d\xi\nonumber\\
&\geq \sum_{k=k_0}^{k_1}\int_{I_k}[1-\cos(\xi)]\cdot |\xi|^{1-2H}d\xi\geq \int_{\frac{(6k_0+1)\pi}{3}}^{\frac{(6k_1+5)\pi}{3}}|\xi|^{1-2H}d\xi\nonumber\\
&=c_H\lt[ \lt(\frac{(6k_1+5)\pi}{3}\rt)^{2-2H} -\lt(\frac{(6k_0+1)\pi}{3}\rt)^{2-2H}  \rt]\nonumber\\
&\geq c_H \lt( \frac{|x-y|}{h} \rt)^{2-2H}.
\end{align}
%\[d_{2, t,h}^2(x,y)\geq c_H t h^{2H}.\]
Thus,  when $h<(t\wedge 1)$,
\[d_{2, t,h}(x,y)\geq  c_H t^{\frac12}h^{H}=c_H t^{\frac12}(h^{H}\wedge   t^{H}).\]
On the interval $\bL=\LL$,
% for $L$ large enough}\footnote{if there is no condition on $|x-y|$, then we don't need this condition}, 
let us select $x_j=jL/t$ for $j=0,\pm 1,\cdots,\pm \lfloor L/t\rfloor$. Similar to the proof of the lower bound part in  Theorem \ref{LUEESup},
%in \cite[Theorem 1.1]{HW2021}, 
we apply the Sudakov minoration
  theorem  (Theorem \ref{Sudakov}) with $\delta=c_H t^{\frac12}|h|^{H}$, which
  yields
  \begin{equation}\label{lower1}
  	\EE\lk\sup_{x\in\bL}\Delta _h u(t,x)\rk\geq \EE\lk\sup_{x_i}\Delta _h u(t,x)\rk\geq c_H t^{\frac12}|h|^{H}\Phi_0(t,L)\,.
  \end{equation}

As a result, combining \eqref{upper1} and  \eqref{lower1}, we accomplish the proof of   \eqref{SupHolderEEasmp}. The proof of \eqref{SupasmpHolder} is similar to that of \eqref{Supasmp}  in Theorem \ref{LUEESup}, by applying Borell's inequality.   This completes the proof of the theorem.
\end{proof}

\bigskip

\begin{proof}[Proof of Theorem \ref{LUtHolderEESup}]
 The  canonical metric associated with the time increment of the solution is defined as
 \[
 d_{3, t, \tau} (x,y)=(\EE|\Delta_\tau u(t,x)-\Delta_\tau u (t,y)|^2)^{\frac 12}\,.
 \]
Recall that 
	\[
    \Delta_\tau u(t,x) =\int_{0}^{t+\tau}\int_{\RR} G_{t+\tau-s}(x-z)W(ds,dz)-\int_{0}^{t}\int_{\RR} G_{t-s}(x-z)W(ds,dz)\,.
	\]
 By the isometric property of the stochastic integral and applying Plancherel's identity with respect to $z$,  we obtain

\begin{align*}
d_{3, t, \tau}^2 (x,y)&=2\int_{\RR_+} [1-\cos(|x-y|\xi)]\cdot \xi^{-1-2H}\Big\{t[1-\cos (|\xi|\tau)]-\frac{\sin(|\xi|\tau)}{2|\xi|}\\
&\quad+\frac{\tau}{2}+\frac{1}{2|\xi|}\lt[\sin(\xi(t+\tau))-\sin(\xi t)\rt]\cdot\lt[\cos(\xi t)-\cos(\xi(t+\tau))\rt]\Big\}d\xi\\
&=2\int_{\RR_+} [1-\cos(|x-y|\xi)]\cdot \xi^{-1-2H}\cdot f_1(t,\tau,\xi)d\xi\\
&\quad+2\int_{\RR_+} [1-\cos(|x-y|\xi)]\cdot \xi^{-1-2H}\cdot f_2(t,\tau,\xi)d\xi\\
&\quad+2\int_{\RR_+} [1-\cos(|x-y|\xi)]\cdot \xi^{-1-2H}\cdot f_3(t,\tau,\xi)d\xi\\
&=:d_{3(1), t, \tau}^2 (x,y)+d_{3(2), t, \tau}^2 (x,y)+d_{3(3), t, \tau}^2 (x,y),
\end{align*}
where
\begin{equation}
\begin{cases}
&f_1(t,\tau,\xi):=t[1-\cos(\tau\xi)], \\
&f_2(t,\tau,\xi):=\frac{\tau}{2}-\frac{\sin(|\xi|\tau)}{2|\xi|}, \\
&f_3(t,\tau,\xi):=-\frac{1}{2|\xi|}\lt[\sin(\xi(t+\tau))-\sin(\xi t)\rt]\cdot\lt[\cos(\xi(t+\tau))-\cos(\xi t)\rt].
\end{cases}
\end{equation}
%It holds $f(t,\tau,\xi)=f_1(t,\tau,\xi)+f_2(t,\tau,\xi)+f_3(t,\tau,\xi)$.

To obtain the upper bound in \eqref{SuptHolderEEasmp}, we first need to estimate the upper bound of $d_{3, t, \tau}^2 (x,y)$. We begin by considering $d_{3(1), t, \tau}^2 (x,y)$. Using the  elementary inequality $1-\cos(x)\leq C_{\theta}x^{2\theta}$ for $\theta\in(0,H)$, we have
\begin{align*}
d_{3(1), t, \tau}^2 (x,y)
%:&=\int_{\RR_+} [1-\cos(|x-y|\xi)]\xi^{-1-2H}f_1(t,\tau,\xi)d\xi\\
&\leq C_{\theta}\cdot t\cdot\tau^{2\theta}\int_{\RR_+} [1-\cos(|x-y|\xi)]\cdot \xi^{-1-2H+2\theta}d\xi\\
&= C_{H,\theta}\cdot t\cdot\tau^{2\theta}|x-y|^{2H-2\theta},
\end{align*}
where in the last equality we change the variable $\xi|x-y|\to\xi$. On the other hand,
\begin{align*}
d_{3(1), t, \tau}^2 (x,y) &\leq C_{H}\cdot t\cdot\int_{\RR_+} [1-\cos(\tau\xi)]\xi^{-1-2H}d\xi\\
&= C_{H}\cdot t\cdot\tau^{2H}.
\end{align*}
Thus, we have
\begin{align}\label{est_d31}
d_{3(1), t, \tau}^2 (x,y)\leq C_{H,\theta}\cdot t\cdot\tau^{2\theta}\lt(|x-y|\wedge \tau\rt)^{2H-2\theta}.
\end{align}
Now we begin to handle $d_{3(2), t, \tau}^2 (x,y)$. Similar to finding the upper bound  of $\cJ_2(t,x,y)$ defined by \eqref{def_J2}, we have
\begin{align}\label{est_d32}
d_{3(2), t, \tau}^2 (x,y)
%:&=\int_{\RR_+} [1-\cos(|x-y|\xi)]\xi^{-1-2H}f_2(t,\tau,\xi)d\xi\\
&=\int_{\RR_+} [1-\cos(|x-y|\xi)]\xi^{-1-2H}\lt(\frac{\tau}{2}-\frac{\sin(|\xi|\tau)}{2|\xi|}\rt)d\xi\nonumber\\
&\leq C_H\tau\cdot\lt(|x-y|^{2H}\wedge \tau^{2H}\rt).
\end{align}
As for the term $d_{3(3), t, \tau}^2 (x,y)$, by the elementary inequality
%$\frac{\sin(|\xi|\tau)}{|\xi|}\leq \tau$ and
 $|\sin(\xi(t+\tau))-\sin(t\xi)|\leq \xi\tau$, it is not hard to see that $|f_3(t,\tau,\xi)|\leq \tau$. Then
\begin{align}\label{est_d33}
d_{3(3), t, \tau}^2 (x,y)
%:&=\int_{\RR_+} [1-\cos(|x-y|\xi)]\xi^{-1-2H}f_3(t,\tau,\xi)d\xi\\
&\leq C_{H}\tau\int_{\RR_+} [1-\cos(|x-y|\xi)]\cdot \xi^{-1-2H}d\xi\nonumber\\
&= C_{H}\tau |x-y|^{2H}.
\end{align}
On the other hand,  since $|\sin(\xi(t+\tau))-\sin(t\xi)|\leq (\xi\tau)\wedge 1$ and $|\cos(\xi(t+\tau))-\cos(t\xi)|\leq (\xi\tau)\wedge 1$, we have 
\begin{align}\label{1est_d33}
d_{3(3), t, \tau}^2 (x,y)
&\leq C_{H}\tau\int_{\RR_+}\xi^{-2-2H}\cdot [(\xi\tau)^2\wedge 1]d\xi\nonumber\\
&= C_{H}\int_0^{1/\tau}\xi^{-2-2H}\cdot (\xi\tau)^2d\xi+C_{H}\int_{1/\tau}^{\infty}\xi^{-2-2H}d\xi\nonumber\\
&\leq C_{H}\tau^{2H+1}.
\end{align}
Thus, 
\begin{align}\label{Est_d33}
d_{3(3), t, \tau}^2 (x,y)
\leq C_{H}\tau(|x-y|\wedge\tau)^{2H}.
\end{align}
Therefore, from \eqref{est_d31}, \eqref{est_d32} and \eqref{Est_d33},  it follows that
\begin{align}\label{d3_upper}
d_{3,t,\tau}^2(x,y)&\leq C_{H,\theta}t\cdot\tau^{2\theta}\lt(|x-y|\wedge \tau\rt)^{2H-2\theta}+C_{H}\tau (|x-y|\wedge\tau)^{2H}.
%&\leq C_{H,\theta}t\cdot\tau^{2\theta}\lt(|x-y|\wedge \tau\rt)^{2H-2\theta}+C_{H}\tau \tcr{(|x-y|\wedge\tau)^{2H}}.
%%&=C_{H,\theta}t\cdot\tau^{2\theta}|x-y|^{2H-2\theta},
%%&{\red =C_{H,\theta}t\cdot\tau^{2\theta}\lt(|x-y|\wedge t\rt)^{2H-2\theta} },
\end{align}
%\tcr{if $|x-y|\leq \tau\leq c( {t}\wedge1)$.}

Now we use Talagrand's majorizing measure theorem (Theorem \ref{Talagrand}) to obtain the upper bound of $\EE \lk\sup_{x\in\bL} \Delta _{\tau} u(t,x)\rk$.  We take 
\[
  A_n(x)=\lt[k\cdot 2^{-2^{n-2}}L,(k+1)\cdot 2^{-2^{n-2}}L\rt)
 \]
and partition $[-L,L]$ as 
\[[-L,L]=\bigcup\limits_{k=-2^{2^{n-2}}}^{2^{2^{n-2}}-1}A_n(x)=\bigcup\limits_{k=-2^{2^{n-2}}}^{2^{2^{n-2}}-1}\lt[k\cdot2^{-2^{n-2}}L,(k+1)\cdot2^{-2^{n-2}}L\rt).\]
The  diameter of $A_n(x)$ with respect to $d_{3,t,\tau}(x,y)$ can be estimated as
\[\diam(A_n(x))\leq C_{H,\theta}t^{1/2}\cdot\tau^{\theta}\lt[\lt(2^{-2^{n-2}}L\rt)\wedge\tau\rt]^{H-\theta}+C_H\tau^{1/2}\cdot[(2^{-2^{n-2}}L)\wedge\tau]^H.\]
Let $N_0=\inf\{n,2^{-2^{n-2}}L\leq \tau\}$. This is, $\log_2(\log_2(L/\tau))+2\leq N_0< \log_2(\log_2(L/\tau))+3$. 
%This gives the upper bound of $d_{3,t,\tau}^2(x,y)$.  In the following, 
Then, by invoking Theorem \ref{Talagrand} in a similar manner to the proof of Theorem \ref{LUEESup}, we obtain when $L\geq \tau$ and $\tau\leq c( {t}\wedge1)$,
\begin{align*}
\EE &\lk\sup_{x\in\bL} \Delta _{\tau} u(t,x)\rk \nonumber\\
  { \leq}& C_H \sup_{x\in\bL} \lk \sum_{n=0}^{N_0}2^{n/2} \diam(A_n(t,x))+ \sum_{n= N_0+1}^{\infty}2^{n/2} \diam(A_n(t,x))\rk \nonumber\\
%  \leq &C_{H,\theta}t^{1/2}\cdot\tau^{\theta}\sum_{n=0}^{N_0}2^{n/2}+C_{H,\theta}t^{1/2}\cdot\tau^{\theta}\sum_{n= N_0+1}^{\infty}(2^{-2^{n-2}}L)^{H-\theta}2^{n/2}\\
%  &+C_H\tau^{1/2+H}\sum_{n=0}^{N_0}2^{n/2}+C_H\tau^{1/2}\sum_{n= N_0+1}^{\infty}(2^{-2^{n-2}}L)^{H}2^{n/2}\\
  \leq &C_{H,\theta} t^{1/2}\cdot\tau^{\theta}\lk \sum_{n= 0}^{N_0}2^{n/2}+\sum_{n= N_0+1}^{\infty}2^{n/2} \lc\frac{2^{2^{N_0-2}}}{2^{2^{n-2}}}\rc^{H-\theta} \rk\\
  &+C_H\tau^{1/2+H}\lk \sum_{n= 0}^{N_0}2^{n/2}+\sum_{n= N_0+1}^{\infty}2^{n/2} \lc\frac{2^{2^{N_0-2}}}{2^{2^{n-2}}}\rc^{H} \rk\\
  \leq &C_{H,\theta}t^{1/2}\cdot\tau^{\theta}2^{N_0/2}+C_H\tau^{1/2+H}2^{N_0/2} \\
  \leq &C_{H,\theta}t^{1/2}\tau^{\theta} \Phi_0(\tau,L).
\end{align*}

%\[
%  \EE \lk\sup_{x\in\bL} \Delta _{\tau} u(t,x)\rk\leq C_{ H,\theta} |\tau|^{\theta} t^{\frac12} \tcr{\Phi_0(\tau,L)} \,,
%  \]
% \tcr{ if $L> T$}.

%Thus, we have \[d_{4,t,\tau}(x,y)\leq C_{H,\theta}t^{\frac12}\tau^{\theta}\cdot\lt(\tau^{H-\theta}\wedge |x-y|^{H-\theta}\rt).\]

Next, we establish the lower bound in \eqref{SuptHolderEEasmp}. To this end, we first need to estimate the upper bound of $d_{3,t,\tau}^2(x,y)$.
%Set $G(\xi):=\lt\{\xi: 0<1-\cos(|x-y|\xi)<\frac12\rt\}$.
%{\blue If $|x-y|\gs t\gs \tau$, consider the same $E(\xi)$ defined by \eqref{E(xi)} with $s$ replaced by $\tau$.} Then same as in the inequalities \eqref{est_I_1} and \eqref{est_intxi},
%{\red \begin{align*}
%d_{4,t,\tau}^2(x,y)&\geq\int_{E(\xi)} [1-\cos(|x-y|\xi)]f_1(t,\tau,\xi)d\xi\\
%&\geq C_H t\cdot\int_{E({\xi})} [1-\cos(\tau\xi)]\xi^{-1-2H} d\xi\\
%&\geq C_H |x-y|^{2H} t\cdot\int_{E(\hat{\xi})} |\hat{\xi}|^{-1-2H} d\hat{\xi}\\
%&= C_{H}t\cdot\tau^{2H}|x-y|^{2H}.
%\end{align*}}
If $|x-y|\geq \tau$ and $\tau \leq c( {t}\wedge1)$, 
then by the inequality  $1-\cos(x)\geq x^2/4$ when $|x|\leq \pi/2$, we have
\begin{align*}
d_{3(1),t,\tau}^2(x,y)&=2\int_{\RR_+} [1-\cos(|x-y|\xi)]\cdot \xi^{-1-2H}\cdot t[1-\cos(\tau\xi)]d\xi\\
&=t|x-y|^{2H}\int_{\RR_+}[1-\cos(\xi)]\cdot \xi^{-1-2H}\cdot \lt[1-\cos\lt(\frac{\tau\xi}{|x-y|}\rt)\rt]d\xi\\
&\geq t \tau^2 |x-y|^{2H-2}\int_{c_0}^{\frac{\pi |x-y|}{2\tau}}[1-\cos(\xi)]\cdot \xi^{1-2H}d\xi,
%&=t \tau^{2H}\int_{\RR_+} \lt[1-\cos\lt(\frac{|x-y|\xi}{\tau}\rt)\rt]\cdot \xi^{-1-2H}\cdot [1-\cos(\xi)]d\xi\\
%&\geq\frac12\int_{E(\xi)} f_1(t,\tau,\xi)d\xi\\
%&=\frac12 t\cdot\int_{E({\xi})} [1-\cos(\tau\xi)]\xi^{-1-2H} d\xi\\
%%&\geq C_H |x-y|^{2H} t\cdot\int_{E(\hat{\xi})} |\hat{\xi}|^{-1-2H} d\hat{\xi}\\
%&\geq C_{H}t\cdot\tau^{2H}.
\end{align*}
where $c_0$ is a fixed positive constant. Using the same estimate as in \eqref{est_devide},  we obtain  
\[\int_{c_0}^{\frac{\pi |x-y|}{2\tau}}[1-\cos(\xi)]\cdot \xi^{1-2H}d\xi\geq c_H  \lt( \frac{|x-y|}{\tau} \rt)^{2-2H}.
 \]
 Thus, we have
 \begin{equation}\label{low_est_d31}
 d_{3(1),t,\tau}^2(x,y)\geq c_H t \tau^{2H}.
 \end{equation}
 For the term $d_{3(2), t, \tau}^2 (x,y)$, we know
 \begin{align*}
d_{3(2), t, \tau}^2 (x,y)
&=\frac{\tau}{2}\int_{\RR_+} [1-\cos(|x-y|\xi)]\xi^{-1-2H}\lt(1-\frac{\sin(|\xi|\tau)}{|\xi|\tau}\rt)d\xi.
\end{align*}
Similar to the lower bound of $d_{2, t,h}^2(x,y)$ in the proof of Theorem \ref{LUHolderEESup}, we can obtain
\begin{equation}\label{low_est_d32}
 d_{3(2),t,\tau}^2(x,y)\geq c_H t \tau^{2H}.
 \end{equation}
For the term $d_{3(3), t, \tau}^2 (x,y)$, by the differential mean value theorem we know
\[|f_3(t,\tau,\xi)|\leq \frac{(\tau \xi\wedge 1)^{\alpha}}{|\xi|}.\] 
Thus, 
\begin{align}\label{upper_est_d33}
d_{3(3), t, \tau}^2 (x,y)&\leq \int_{\RR_+} [1-\cos(|x-y|\xi)]\cdot \xi^{-2-2H}\cdot (\tau \xi\wedge 1)^{\alpha}d\xi\nonumber\\
&=\tau^{1+2H}\int_{\RR_+} \lt[1-\cos\lt(\frac{|x-y|\xi}{\tau}\rt)\rt]\cdot \xi^{-2-2H}\cdot (\xi\wedge 1)^{\alpha}d\xi\nonumber\\
&=\tau^{1+2H}\int_0^1\lt[1-\cos\lt(\frac{|x-y|\xi}{\tau}\rt)\rt]\cdot \xi^{-2-2H+\alpha}d\xi\nonumber\\
&\qquad+\tau^{1+2H}\int_1^{\infty}\lt[1-\cos\lt(\frac{|x-y|\xi}{\tau}\rt)\rt]\cdot \xi^{-2-2H}d\xi\nonumber\\
&\leq C_{H,\alpha}\tau^{1+2H}+C_H\tau^{1+2H}.
\end{align}
%\tcr{By the inequality $\sin(|\xi|\tau)\leq |\xi|\tau$, we have}
%\[|f_2(t,\tau,\xi)+f_3(t,\tau,\xi)|\leq \tau.\] Then 
Since $f_1(t,\tau,\xi)+f_2(t,\tau,\xi)+f_3(t,\tau,\xi)\geq f_1(t,\tau,\xi)+f_2(t,\tau,\xi)-|f_3(t,\tau,\xi)|$, combining \eqref{low_est_d31}, \eqref{low_est_d32} and \eqref{upper_est_d33}, we have
\begin{align*}
d_{3,t,\tau}^2(x,y)&\geq c_H t \tau^{2H}-C_H\tau^{1+2H} \\
&\geq c'_{H} t\tau^{2H}
\end{align*}
if $\tau$ is small enough. Hence, if $|x-y|\geq \tau$, it holds that
\begin{equation}\label{d3_lowerupper}
d_{3,t,\tau}(x,y)\geq c'_H t^{1/2}\tau^{H} .
\end{equation}
This gives the lower bound of $d_{3,t,\tau}(x,y)$.

Now, we apply the Sudakov minoration theorem to obtain the lower bound in \eqref{SuptHolderEEasmp}. On the interval $\bL=\LL$, for sufficient large $L$, we select $x_j=jL/\tau$ for $j=0,\pm 1,\cdots,\pm \lfloor L/\tau\rfloor$. Similar to the proof of the lower bound in the first part of Theorem \ref{LUEESup}, applying the Sudakov minoration
  theorem  (Theorem \ref{Sudakov}) with $\delta=c_Ht^{1/2} \tau^{H}$
  yields
  \begin{equation*}
  	\EE\lk\sup_{x\in\bL}\Delta _{\tau} u(t,x)\rk\geq \EE\lk\sup_{x_j}\Delta _{\tau} u(t,x)\rk\geq c_H t^{1/2}\tau^{H}\Phi_0(\tau,L)\,.
  \end{equation*}
  This completes the proof of  \eqref{SuptHolderEEasmp} in this theorem. 

Analogous to the second part of Theorem \ref{LUEESup}, \eqref{SuptasmpHolder} can also be established. This completes the proof.
\end{proof}

\appendix
\section{Lemmas used in proofs}\label{appenA}
\begin{Lem}\label{equiv}
  If the process $\{X_{t},t\in T\}$ is symmetric,   then we have
  \begin{equation}
    \EE\big[\sup_{t\in T} |X_{t}|\big]\leqslant 2\EE\big[\sup_{t\in T} X_{t} \big]+\inf_{t_0\in T}\EE\big[|X_{t_0}|\big]\,.
    %\leqslant 3\EE\big[\sup_{t\in T} |X_{t}|\big].
  \end{equation}
\end{Lem}

\begin{Thm} {\rm (Talagrand's majorizing measure theorem,  see  e.g.
   \cite[Theorem 2.4.2]{Talagrand2014}).} \label{Talagrand}
 For fixed $T>0$, let $\{X_t,t\in T\}$ be a centered Gaussian process indexed by $T$.  Denote
 	 by $d(t,s):=(\EE|X_t-X_s|^2)^{\frac 12}$  
 	     the associated  %\tcb
	     {canonical metric} of $X_t$ on $T$.
 	 %  Given $\alpha>0$, %and a metric space $(T,d)$
% 	 we define
% \begin{eqnarray}
%    \gamma_\alpha(T,d)
%    &:=&\inf \sup_{t}\sum_{n\geq 0}2^{n/\alpha} \diam(A_n(t))\,,
%\label{Gamma_alpha}
%%\\
%%    & =& \inf \sup_{t\in T}\sum_{n\geq 0} 2^{n/2} d(t,\TT_n),
%%     \nonumber
%  \end{eqnarray}
%  and in the second equation the infimum is taken over all sequences of sets $\TT_n$ with cardinality $|\TT_n|<2^{2^n}$.
%
  Then
 	\begin{equation}
 	\EE\Blk \sup_{t\in T} X_t \Brk\approx \gamma_2(T,d):=\inf_\cA \sup_{t\in T} \sum_{n\geq 0}2^{n/2} \diam(A_n(t))\,,
 	\end{equation}
 	where   the infimum is taken over all increasing sequence $\cA:=\{\cA_n, n=1, 2, \cdots\}$ of partitions of $T$ such that $\#\cA_n  \leq 2^{2^n}$ ($\#A$ denotes the number of elements in the  set $A$),  where $A_n(t)$ denotes the unique element of $\cA_n$ that contains $t$, and $\diam(A_n(t))$ is the diameter of $A_n(t)$.
%    where the infimum is taken over all admissible sequences.
 \end{Thm}

 \begin{Thm}{\rm (Sudakov minoration  theorem,  see   e.g.  \cite[Lemma 2.4.2]{Talagrand2014}).} \label{Sudakov}
		Let $\{X_{t_i},i=1,\cdots,L\}$ be a centered Gaussian family with  %\tcb
		{canonical metric $d(t,s):=(\EE|X_t-X_s|^2)^{\frac 12}$. Suppose there exists a finite subset $\{t_1,t_2,\cdots, t_L\}\subset T$ such that for all $p\neq q$, 
		\[
		 d(t_p,t_q)\geq \delta.
		\]}
	Then, we have
		\begin{equation}
		\EE\Blc\sup_{1\leq i\leq L} X_{t_i} \Brc \geq \frac{\delta}{C} \sqrt{\log_2(L)},
		\end{equation}
		where $C$ is a universal constant.
	\end{Thm}
\begin{Thm} {\rm (Borell's inequality, see   e.g.   \cite[Theorem 2.1]{adler}).}
\label{Borell} { Let $\{X_t,t\in T\}$ be a centered separable   Gaussian process on some topological index set $T$
 with almost surely bounded sample paths.}   Then    \[\EE\Blc\sup_{t\in  T} X_t\Brc<\infty ,\]
 and for all $\lambda>0$
	\begin{equation}
	\bP\lt\{\lt|\sup_{t\in T} X_t-\EE\Blc\sup_{t\in T} X_t\Brc\rt|>\lambda\rt\}\leq 2\exp\lc-\frac {\lambda^2}{2\sigma_T^2}\rc,
	\end{equation}
	where $\sigma_T^2:=\sup_{t\in T}\EE(X_t^2)$.
 \end{Thm}

\section{Auxiliary proofs in Section \ref{s.2}}\label{appenB}

\subsection{Proof of \eqref{eq:sin-1F2}}\label{appenB.1}
{%\red 
In this section, we show that \eqref{eq:sin-1F2} holds. By      Eq. (3.761.1) in \cite{GRBook2015}, we have
\begin{equation}\label{eq:3.761.1}
\int_{0}^{1} x^{\mu-1} \sin(a x)\, dx
= \frac{-i}{2\mu}
\bigl[{}_1F_1(\mu;\mu+1;ia)-{}_1F_1(\mu;\mu+1;-ia)\bigr]
\end{equation}
as long as $a>0,\; \Re\mu>-1,\; \mu\neq 0.$ Here, 
\begin{align*}
	{}_1F_1(\mu; \mu+1; z) = \sum_{n=0}^{\infty} \frac{(\mu)_n}{(\mu+1)_n} \frac{z^n}{n!}\,,
\end{align*}
where $(\mu)_0=1$ and $(\mu)_n=\mu(\mu+1)(\mu+2)\cdots(\mu+n-1)$ for $n\geq 1$.
Taking the difference between the two hypergeometric series, we obtain
\begin{align*}
	D(\mu,z) = \sum_{n=0}^{\infty} \frac{(\mu)_n}{(\mu+1)_n} \frac{z^n - (-z)^n}{n!}.
\end{align*}
When $n$ is even, $z^n - (-z)^n = 0$. When $n$ is odd, $z^n - (-z)^n = 2z^n$. Setting $n=2k+1$ yields
\begin{align*}
	D(\mu,z) &= \sum_{k=0}^{\infty} \frac{(\mu)_{2k+1}}{(\mu+1)_{2k+1}} \frac{2z^{2k+1}}{(2k+1)!} \\
	&= \sum_{k=0}^{\infty} \frac{a}{a+2k+1} \frac{2z^{2k+1}}{(2k+1)!} \\
	&= 2az \sum_{k=0}^{\infty} \frac{1}{a+2k+1} \frac{z^{2k}}{(2k+1)!}\,.
\end{align*}
Substituting $\mu = 2H_0-2$ and $z = i\rho$ into the above new series representation for the difference $D(\mu,z)$, we have
\begin{align*}
    	{}_1F_1(2H_0-2; 2H_0-1; i\rho) &- {}_1F_1(2H_0-2; 2H_0-1; -i\rho) \\
    	&= 2(2H_0-2)(i\rho) \sum_{k=0}^{\infty} \frac{1}{(2H_0-2)+2k+1} \frac{(i\rho)^{2k}}{(2k+1)!} \\
    	&= 4i\rho(H_0-1) \sum_{k=0}^{\infty} \frac{1}{2H_0+2k-1} \frac{(-1)^k \rho^{2k}}{(2k+1)!}\,.
    \end{align*}
Putting this back into the expression \eqref{eq:3.761.1}, we get
\begin{align*}
	I(\rho, H_0) &= \rho \sum_{k=0}^{\infty} \frac{(-1)^k \rho^{2k}}{(2k+1)! (2H_0+2k-1)}\\
	&= \frac{\rho}{2H_0-1} \sum_{k=0}^{\infty} \frac{(H_0-\frac{1}{2})_k}{(\frac{3}{2})_k (H_0+\frac{1}{2})_k} \frac{(-\rho^2/4)^k}{k!} \\
	&=\frac{\rho}{2H_0-1} {}_1F_2\left(H_0-\frac{1}{2}; \frac{3}{2}, H_0+\frac{1}{2}; -\frac{\rho^2}{4}\right),
\end{align*}
which establishes \eqref{eq:sin-1F2}.
}

\subsection{Alternative proof of \eqref{eq:1F2_infty}}\label{appenB.2}

{%\red 
We next present an alternative derivation of \eqref{eq:1F2_infty}.  
According to \cite[Eq. 16.5.1]{OF2010}, the generalized hypergeometric function ${}_1F_2(a_1; b_1, b_2; z)$ admits the Mellin-Barnes representation
\begin{align}\label{eq:MB-1F2}
	{}_1F_2(a_1; b_1, b_2; z) = \frac{\Gamma(b_1)\Gamma(b_2)}{\Gamma(a_1)} \frac{1}{2\pi i} \int_{\mathcal{C}} \frac{\Gamma(a_1+s)\Gamma(-s)}{\Gamma(b_1+s)\Gamma(b_2+s)}(-z)^s ds\,, 
\end{align}
where the contour of integration separates the poles of $\Gamma(a_1+s)$ from the poles of $\Gamma(-s)$. Recall that 
\begin{align*}
	a_1=H_0-\tfrac12 , \quad b_1=\tfrac32, \quad b_2=H_0+\tfrac12\,, \quad z = -\frac{\rho^2}{4}\,.
\end{align*}
Given $H_0 \in (1/2, 1)$, we have $a_1 \in (0, 1/2)$, so we may choose any $c \in (-1/2, 0)$ such that the contour $\mathcal{C}$ runs from $c-i\infty$ to $c+i\infty$. To determine the asymptotic behavior for large $\rho$, we close the contour to the left by a large semicircle and apply Cauchy's residue theorem to \eqref{eq:MB-1F2}. This gives 
\begin{align*}
	{}_1F_2(a_1; b_1, b_2; z) \approx \frac{\Gamma(b_1)\Gamma(b_2)}{\Gamma(a_1)} \sum_{n=0}^\infty \text{Res}_{s = -a_1-n} \frac{\Gamma(a_1+s)\Gamma(-s)}{\Gamma(b_1+s)\Gamma(b_2+s)}(-z)^s \,,
\end{align*}
as $-z=\frac{\rho^2}{4} \to \infty$. Clearly, the leading-order term arises from the residue at the rightmost pole $s_0 = -a_1=-H_0+\tfrac12$.  
Evaluating this residue gives
\begin{align*}
{}_1F_2(a_1; b_1, b_2; z) &\approx \frac{\Gamma(b_1)\Gamma(b_2)}{\Gamma(a_1)} \left[ \frac{\Gamma(a_1)}{\Gamma(b_1-a_1)\Gamma(b_2-a_1)}(-z)^{-a_1} \right] \\
&= \frac{\Gamma(b_1)\Gamma(b_2)}{\Gamma(b_1-a_1)\Gamma(b_2-a_1)}(-z)^{-a_1}\\
&=\frac{\frac{\sqrt{\pi}}{2} \Gamma(H_0 + \frac{1}{2})}{\Gamma(2 - H_0)} \cdot 2^{2H_0 - 1} \rho^{-2H_0 + 1}
\end{align*}
as $\rho$ approaches infinity. This completes the proof of \eqref{eq:1F2_infty}.
}

\bigskip
\bigskip
\noindent\textbf{Acknowledgement}: The authors thank Professor Samy Tindel for his helpful comments. SL is supported by the Research Grants Council of Hong Kong (Grant No. P0031382/S-ZG9U) and the Postdoc Matching Fund of PolyU (Grant No. 1-W32B). YH was supported by the NSERC discovery fund and a startup fund of the University of Alberta. XW was supported by the startup fund of Sun Yat-sen University and the research fund of Johns Hopkins University.

\bibliographystyle{amsplain}
\bibliography{Ref_SWE_Additive}

\end{document}